\documentclass[11pt]{article}
\usepackage{amsthm, amssymb, srcltx}
\usepackage{soul, color}
\usepackage[]{amsmath}
\usepackage[]{amsfonts}
\usepackage[]{fancyhdr}
\usepackage[]{graphicx}
\graphicspath{{/EPSF/}{../figures/}{figures/}}

\makeatletter

\@addtoreset{equation}{section}
\makeatother

\newtheorem{theorem}{Theorem}[section]
\newtheorem{lemma}[theorem]{Lemma}
\newtheorem{proposition}[theorem]{Proposition}

\newtheorem{remark}[theorem]{Remark}
\newtheorem{hypothesis}[theorem]{Hypothesis}

\def \Rm {\mathbb{R}}

\def\pa{\partial}

\def\C{\mathcal{C}}

\def\L{\mathcal{L}}

\def\diam{\Delta}

\newcommand{\pdr}[2]{\dfrac{\partial{#1}}{\partial{#2}}}
\newcommand{\pdrr}[2]{\dfrac{\partial^2{#1}}{\partial{#2}^2}}

\newcommand{\pdrt}[3]{\dfrac{\partial^2{#1}}{\partial{#2}{\partial{#3}}}}

\newcommand{\yperp}{y\cdot\hat\theta^\perp}
\newcommand{\xiperp}{\xi\cdot\hat\theta^\perp}
\newcommand{\where}{\quad\text{ where }}
\newcommand{\qandq}{\quad\text{ and }\quad}
\newcommand{\qhenceq}{\quad\text{ hence }\quad}

\newcommand{\cout}[1]{}

\newcommand{\mC}{\mathcal{C}}
\newcommand{\mF}{\mathcal{F}}

\newcommand{\mK}{\mathcal{K}}
\newcommand{\mL}{\mathcal{L}}

\newcommand{\mZ}{\mathcal{Z}}

\newcommand{\mD}{\mathfrak D}

\newcommand{\sgn}[1]{\,{\rm sign}(#1)}

\newcommand{\olL}{{\overline{\|L\|}}}

\def \nsphere{ {\mathbb{S}^{n-1}}}

\def \Sone{ {\mathbb{S}^1} }

\def \supp { {\mbox{supp}}}

\def \S { {\mathbb{S}}}

\def \bfx{ {\mathbf{x}} }

\title{Inverse transport with isotropic time-harmonic sources}

\author{Guillaume Bal\thanks{gb2030@columbia.edu} \and Fran\c cois Monard\thanks{fm2234@columbia.edu \newline Department of Applied Physics and Applied Mathematics, Columbia University, New York NY, 10027}}

\begin{document}
\maketitle

\begin{abstract}
  This paper concerns the reconstruction of the scattering coefficient in a two-dimensional transport equation from angularly averaged measurements when the probing source is isotropic and time-harmonic. This is a practical setting in the medical imaging modality called Optical Tomography. As the modulation frequency of the source increases, we show that the reconstruction of the scattering coefficient improves. More precisely, as the frequency $\omega$ increases, we show that all frequencies of the scattering coefficient lower than $b$ are reconstructed stably with an accuracy that improves as $\omega$ increases and $b$ decreases. The proofs are based on  an analysis of the single scattering singularities of the transport equation and on careful analyses of oscillatory integrals by stationary phase arguments.
\end{abstract} 

\section{Introduction}

In the theory of stationary inverse transport, it was shown in \cite{BLM} that, with knowledge of the attenuation coefficient, the problem of reconstructing the spatial part of the scattering coefficient $k(x)$ from isotropic sources and angularly averaged measurements was severely ill-posed. This was due to the fact that the measurements were not capturing enough singularities of the transport solution because of angular averaging. 
In the present paper, we consider a more favorable measurement setting where the (still isotropic) source term is modulated in frequency with a sufficiently high modulation frequency. For sufficiently high frequencies, this corresponds to measurements of the time-dependent transport equation at very short time sampling (theory for this can be found in \cite{BJ}). In either case, asymptotic expansions (for large modulation frequencies or small times) of the single scattering component of the measurements show that in the leading order, and in the case where the domain of interest is a ball, we recover a weighted X-Ray transform of the scattering coefficient \cite{BJ, BJLM}, while the remainder term in the measurements is of lower-order. The practical interest of using time-harmonic sources in transport for improving the reconstruction of optical parameters was first pointed out numerically in \cite{RBH} in the context of least-squares based reconstructions, where the authors observed that cross-talk between reconstructed optical coefficients was reduced and accuracy of the reconstructions was improved as $\omega$ increased. One of the goals of the present paper is to provide theoretical justifications for these observations.

Although the remainder in the asymptotic expansion of forward transport solutions has small magnitude (see \cite{BJLM}), it can still be magnified during the inversion procedure because the inverse X-Ray transform is a deregularizing operator. In order to avoid this, we need to regularize the inversion procedure. We do so by giving up on the reconstruction of the high-frequency content of the unknown scattering coefficient $k$. Our main objective is to show which spatial frequencies of $k$ may be stably reconstructed and with which accuracy as a function of the modulation frequency $\omega$.

The task of the present paper is to devise a proper reconstruction algorithm for the scattering coefficient in this setting that is justified theoretically. The theory is based on stationary phase arguments, which allow us to exhibit decay of solutions (see also \cite{BJLM}) and errors in reconstructions in terms of the modulation frequency $\omega$. The same techniques are also used to guarantee the convergence of an iterative reconstruction scheme and in particular to prove that certain error operators become contractions when $\omega$ becomes large enough.

The results presented here show that in order to attain a given precision on the reconstruction, there is a trade-off between the modulation frequency $\omega$ of the measurements available, and the bandwidth $b$ (or fineness of details) at which one will recover the function $k$. Knowledge of measurements at a given frequency $\omega$ allow for a stable reconstruction of the scattering coefficient up to a maximal resolution $b$ that depends on this $\omega$ and the accuracy we wish to attain. Modulated sources at a given frequency corresponds to a spatial scale below which details of $k$ are not accessible. We make such statements precise. 

The rest of the paper is structured as follows. We state our main results in section \ref{sec:statement}. Section \ref{sec:forward} recalls the necessary elements of forward transport theory in a general setting, including the asymptotic behavior of angularly averaged solutions of time-harmonic transport for high modulation frequencies. Section \ref{sec:inversion} describes how to invert for the spatial part of scattering $k(x)$ in two dimensions of space and when the domain is a disk. Sections \ref{sec:proof1} and \ref{sec:proof2} provide proofs for the main results. 

\section{Statement of the main results}\label{sec:statement}
Let $r>0$ and let $B_r$ denote the open disk $\{ x_1^2 + x_2^2< r^2\}$ in $\Rm^2$ with boundary denoted by $\partial B_r$. We consider the two-dimensional time-harmonic transport equation on $B_r\times\Sone$:
\begin{align}
    \begin{split}
	v\cdot\nabla u + (\sigma(x) + i\omega) u &= k(x) \int_\Sone \phi(v',v) u(x,v')\ dv', \quad (x,v)\in X\times\Sone, \\
        u(x,v) &= \delta(x-x_0), \quad (x,v)\in\Gamma_{-},	
    \end{split}    
    \label{eq:thtransport2d}
\end{align}
where $\delta(x-x_0)$ is an isotropic point source located at $x_0$ and is defined for any test function $\psi$ on $\partial B_r$ by
\begin{math}
    \int_{\partial B_r} \delta(x-x_0) \psi(x)\ d\mu(x) = \psi(x_0).
\end{math}

In this paper, we consider {\em angularly averaged outgoing measurements} given by
\begin{align}
    T^\omega(x_0,x_c):= \int_{v\cdot\nu_{x_c}>0} u_{|\Gamma_+}(x_c,v) \,v\cdot\nu_{x_c}\ dv, \qquad u \text{ solves } \eqref{eq:thtransport2d}.
    \label{eq:measurements}
\end{align}
Such measurements are reasonable descriptions of the data acquired in practice in the medical imaging modality called Optical Tomography; see \cite{RBH} and references there. Because the source term in \eqref{eq:thtransport2d} is isotropic and the measurements are angularly averaged, the main singularities of the transport equation are not captured by such measurements. The reconstruction of the optical coefficients thus becomes a severely ill-posed problem \cite{B2,BLM}. One way to restore some well-posedness in the inversion is to consider the limit as $\omega\to\infty$. This is the problem considered in this paper following results obtained in \cite{BJLM} on the forward transport problem.

Under appropriate subcriticality conditions recalled in the next section, we obtain existence and uniqueness of the solution of \eqref{eq:thtransport2d} as well as the decomposition 
\begin{align}
    u(x,v) = J_\omega\delta_{x_0} + \mK_\omega J_{\omega}\delta_{x_0} + (I-\mK_\omega)^{-1} \mK_\omega^2 J_{\omega}\delta_{x_0},
    \label{eq:sol_decomp}
\end{align}
where the operators $J_\omega$ and $\mK_\omega$ are defined in \eqref{Jom} and \eqref{eq:Kom} below, respectively. This decomposition in turn yields a decomposition of the measurement operator \eqref{eq:measurements} into three components denoted by
\begin{align}
    T^\omega[\sigma,k](x_0,x_c):= T_0^\omega[\sigma](x_0,x_c) + T_1^\omega[\sigma,k](x_0,x_c) + T_{2+}^\omega[\sigma,k](x_0,x_c),
    \label{eq:meas_decomp}
\end{align}
where $T_0$ accounts for particles that traveled from $x_0$ to $x_c$ in a ballistic way and $T_1$ and $T_{2+}$ account for particles that were emitted at $x_0$, scattered once or multiple-times inside the domain, respectively, and were measured at $x_c$. Here and below, $x_0$ denotes the emitter's position and $x_c$, the detector's position.

In this paper, as in \cite{BLM}, we assume $\sigma$ to be known and focus on the reconstruction of the scattering coefficient $k(x)$. The ballistic part can be taken out of the measurements and we define our data to be
\begin{align}
    \mD^\omega[k](x_0, x_c) := T^\omega[\sigma,k](x_0, x_c) - T_0^\omega[\sigma](x_0, x_c) = T_1^\omega[k](x_0,x_c) + T_{2+}^\omega[k](x_0,x_c).
    \label{eq:data}
\end{align}
In the sequel, we address the reconstruction of the function $k$ from knowledge of the data \eqref{eq:data}.

Throughout the paper, we require the following crucial hypothesis, which expresses the fact that $k$ vanishes in the vicinity of $\partial X$:
\begin{hypothesis}\label{hyp:supportk}
    There exists $0<D<r$ such that $\supp\ k \subset B_{r-2D}$.
\end{hypothesis}
The reason is that the measurements we consider are overwhelmed by a term that depends only on the value of $k$ at the domain's boundary. In other words, to leading order, the measurements do not depend on the scattering coefficient $k$ away from the boundary. We thus assume here hypothesis \ref{hyp:supportk}, which in practice corresponds to placing the detectors away from the scattering domain of interest and thus does not seem to be overly restrictive.

Now, as is proved in \cite{BJLM} assuming a hypothesis equivalent to \ref{hyp:supportk}, the expression \eqref{eq:data} admits the following asymptotic decomposition for large $\omega$:
\begin{align}
    \begin{split}
	\mD^\omega(x_0, x_c) &= \widetilde{T_1^\omega}[k](x_0, x_c) + T_{1R}^\omega[k](x_0, x_c) + T_{2+}^\omega[k](x_0, x_c), \where \\[2mm]
	\widetilde{T_1^\omega}[k](x_0, x_c) &= \frac{1}{\sqrt{\omega}} A^\omega (x_0,x_c) \int_{[x_0,x_c]} k \rho, \qandq \| T_{1R}^\omega[k] + T_{2+}^\omega[k] \|_{L^\infty( (\partial X)^2)} \le \frac{C}{\omega},	
    \end{split}
    \label{eq:data_asym}    
\end{align}
and where $A^\omega$ is bounded away from zero independently of $\omega$ and $\rho(x) = (r^2-|x|^2)^{-\frac{1}{2}},\ x\in B_r$. In other words, the single scattering decomposes into a leading part $\widetilde{T_1^\omega}$ that is proportional to the X-Ray transform of $k\rho$ (and from which we would like to reconstruct all or part of $k$) and a remainder $T_{1R}^\omega$ that is asymptotically smaller. Equality \eqref{eq:data_asym} tells us that for large $\omega$, the remainder $T_{1R}^\omega + T_{2+}^\omega$ becomes negligible.

In order to setup an inversion, we first reparameterize the data into the parallel geometry of the X-Ray transform, namely we define $x_0\in\partial B_r$ and $x_c\in\partial B_r$ as functions of $(s,\theta)\in(-r,r)\times\Sone$, such that the line joining $x_0$ and $x_c$ is exactly 
\begin{align*}
    L(s,\theta) := \left\{ s\hat\theta^\perp + t\hat\theta, \quad t\in\Rm \right\},
\end{align*}
and rewritting equality \eqref{eq:data_asym} in this parametrization yields
\begin{align}
    \chi(s) \mD^\omega(s,\theta) = \frac{1}{\sqrt{\omega}} A^\omega (s,\theta) P\left[ \rho k \right](s,\theta) + \frac{1}{\omega} \chi(s) R^\omega(s,\theta),
    \label{eq:data_asym_radon}
\end{align}
where $P$ denotes the X-Ray transform in the parallel geometry, and $\chi$ is a smooth function satisfying $\chi(s) = 1$ if $|s|\le r-2D$ and $\chi(s)=0$ if $|s|\ge r-D$. As will be seen later, $\chi$ is necessary in order to smooth out effects at the boundary of the term $R^\omega$, while leaving the term $P\left[ \rho k \right]$ unchanged thanks to hypothesis \ref{hyp:supportk}.

Let us now set up an inverse for $\widetilde{T_1^\omega}[k]$. The first candidate for an inversion is 
\begin{align*}
    \widetilde{T_1^\omega}^{-1} = \sqrt{\omega} P^{-1} [(A^\omega)^{-1} \chi \cdot],
\end{align*}
where $P^{-1}$ denotes the inverse X-Ray transform. Such an inverse is unsuitable because $P^{-1}$ is a deregularizing operator that prevents us from controlling the error term $\widetilde{T_1^\omega}^{-1}\circ [T_{1R}^\omega + T_{2+}^\omega]$ in a suitable manner. We therefore introduce a regularized inverse X-Ray transform $P^{-1,b}$ with $b>0$, such that 
\begin{align}
    P^{-1,b}\circ  P [f] (x) = W_b\star f (x):= f_b(x), \where\quad W_b(x):= \frac{1}{2\pi} \int_{\Rm^2} e^{ix\cdot\xi} \hat\Phi \left( \frac{|\xi|}{b} \right)\ d\xi
    \label{eq:lowpass}
\end{align}
with $\hat\Phi:[0,\infty)\mapsto [0,1]$ a ``low-pass filter'' supported inside $[0,1]$. Hence, the parameter $\frac{1}{b}$ measures the finest scale at which we reconstruct $f$ when applying $P^{-1,b}$ to $P[f]$. The inverse operator that we apply to the data \eqref{eq:data} is therefore defined by
\begin{align}
    \widetilde{T_1^\omega}^{-1,b} := \sqrt{\omega} P^{-1,b}[(A^\omega)^{-1} \chi \cdot].
    \label{eq:inversion_operator}
\end{align}
As the operator \eqref{eq:inversion_operator} is applied to the data \eqref{eq:data}, we obtain an equation of the form
\begin{align}
    \widetilde{T_1^\omega}^{-1,b} \mD^\omega[k] = [k\rho]_b + R^{\omega,b}[k], 
    \label{eq:inversion_equation}
\end{align}
where the ``error'' term $R^{\omega,b}[k]$ is decomposed into $R^{\omega,b} = R_1^{\omega,b} + R_2^{\omega,b}$, with
\begin{align}
    R_1^{\omega,b} [k] &:= \widetilde{T_1^\omega}^{-1,b} T_{1R}^\omega[k] = \widetilde{T_1^\omega}^{-1,b} T_1^\omega[k] - [k\rho]_b, \label{eq:rem1} \\
    R_2^{\omega,b} [k] &:= \widetilde{T_1^\omega}^{-1,b} T_{2+}^\omega[k]. \label{eq:rem2}
\end{align}
Here, $R_1^{\omega,b}[k]$ represents the error due to the part of the single scattering that is not inverted for whereas $R_2^{\omega,b}[k]$ represents the error due to multiple scattering. 

Our first result controls the remainder $R^{\omega,b}[k]$ in $L^\infty(B_r)$ using the asymptotic decomposition for large $\omega$ of the measurements \eqref{eq:data_asym} as established in \cite{BJLM}:
\begin{theorem}[Direct inversion]\label{thm:straight_inversion}
    Let $\sigma\in\C^2(B_r)$ be known and assume that $k\in\C^2(B_r)$ and that the support hypothesis \ref{hyp:supportk} holds. Then for $\omega$ large enough and fixed $b$,  knowledge of the data $\mD^\omega(s,\theta)$ for $s\in[-r+D,r-D]$ and $\theta\in\Sone$ allows us to reconstruct $[\rho k]_b$ up to an error $R^{\omega,b}[k]$ \eqref{eq:inversion_equation} that is bounded in $L^\infty(B_r)$ by 
    \begin{align*}
	\|R^{\omega,b}[k]\|_{L^\infty} \le C\frac{b}{\sqrt{\omega}}.
    \end{align*}
    Here, the constant $C$ depends on $\|k\|_{\C^2}$, $\|\sigma\|_{\C^2}$, $D$ and $\|w_1\|_{L^1}$.
\end{theorem}

\begin{remark}\label{rmk:straight_inversion}
    Theorem \ref{thm:straight_inversion} merely states that in order to get details of size $\frac{1}{b}$ while keeping a constant error level, we must provide measurements at frequency $\omega \approx b^2$. Conversely, if measurements are available at a given frequency $\omega$, then theorem \ref{thm:straight_inversion} gives the increase in error as we try to recover finer details of $[\rho k]$.    
\end{remark}

In the second result, which is the main result of the paper, we assess the conditions under which the following iterative algorithm:
\begin{align}
    \begin{split}
	k_0 &= \widetilde{T_1^\omega}^{-1,b} \mD^\omega, \\
	k_{n+1} &= G(k_n) := \widetilde{T_1^\omega}^{-1,b} \mD^\omega - R^{\omega,b}[k_n], \quad n\ge 0,
    \end{split}    
    \label{eq:iterated_scheme}
\end{align}
improves the reconstruction of $[k\rho]_b$ for a given $b$.
The reconstruction in Theorem \ref{thm:straight_inversion} treats multiple scattering as an unknown error term. However, once an approximation of $k$ is obtained, we can estimate the multiple scattering contribution to the data and remove it from the data is order to obtain a better approximation of the single scattering contribution on which the reconstruction algorithm is based. Improvements are then obtained iteratively as described in \eqref{eq:iterated_scheme}. 

Our objective is to show the improvements resulting from using such an iterative scheme and to obtain conditions that guaranty its convergence. In particular, we must find a suitable space for $k$ and a $(b,\omega)$-regime where the operator $G$ is a $c_1$-contraction for $0<c_1<1$. This in turn requires a thorough and somewhat surprisingly complicated analysis of the operators $R_1^{\omega,b}$ \eqref{eq:rem1} and $R_2^{\omega,b}$ \eqref{eq:rem2} that  includes a careful analysis of oscillatory integrals by the method of stationary phase with errors that depend on the function $k$ but not on its derivatives. 

The following theorem states that such a regime exists and details the accuracy that is reached in this case. In particular, we emphasize the fact that no regularity conditions are imposed on $k$.
\begin{theorem}[Iterative improvement]\label{thm:iterative}
    Assume that $\sigma\in\C^2(B_r)$ is known and that $k$ satisfies hypothesis \ref{hyp:supportk}. Suppose further that the function $W_1$ in \eqref{eq:lowpass} satisfies $\|W_1\|_{L^1(\Rm^2)}<\infty$. Then for fixed $b>0$, there exists $K_1>0$ and $\omega_0>1$ such that for $\omega\ge\omega_0$ and $k$ such that $\|k\rho\|_\infty\le K_1$, the iterative scheme \eqref{eq:iterated_scheme} converges to $k^\star \in B_{K_0}(L^\infty(B_{r-2D}))$. Moreover $k^\star$ satisfies the error estimate
    \begin{align}
	\|k^{\star} - [k\rho]_b\|_\infty \le \frac{c_1}{1-c_1} \|[k\rho]-[k\rho]_b \|_\infty,	
	\label{eq:iterative_error_estimate}
    \end{align}
    where $c_1\in (0,1)$.
\end{theorem}

For fixed $b$, the constant $c_1$ behaves in $\omega$ like 
\begin{align*}
    c_1 \approx \frac{C_1}{\omega} (b^2 + b^5) +  C_2 \frac{b^3}{\omega^{\frac{1}{2}}}\log\left( \frac{\omega}{b} \right),
\end{align*}
where the first and second terms in the right-hand side are constants of boundedness in $L^\infty(B_{r-2D})$ for the error operators $R_1^{\omega,b}$ and $R_2^{\omega,b}$, respectively.

\begin{remark}
    Theorem \ref{thm:iterative} implies a much better accuracy at the cost of a more expensive inversion because of the iterative scheme. At each iteration, the computation of $R^{\omega,b}$ requires the computation of a forward transport problem. 
\end{remark}

\begin{remark}[On generalizations to higher dimensions]
The extension of Theorem \ref{thm:straight_inversion} to higher dimensions should cause no difficulty as the proof of this theorem mostly relies on the asymptotic expansion of the forward solutions derived in \cite{BJLM} for all dimensions. The extension of  Theorem \ref{thm:iterative} should prove much more challenging. The fixed point method applied in Theorem \ref{thm:iterative} requires that the stationary phase integrals compensate the inversion of a term linear in $k$ whose amplitude of order  $\omega^{\frac{n-1}{2}}$ grows with spatial dimension $n$. In dimension $n=2$, the stationary phase integrals provide $O(\omega^{-1})$ contributions that compensate the previous term. The difficult problem of how the stationary phase techniques, which are at the core of our proof of Theorem \ref{thm:iterative}, should be modified in higher spatial dimensions is not considered here.
\end{remark}

\paragraph{Roadmap of proofs :}

Theorem \ref{thm:straight_inversion} almost comes for free once the results from \cite{BJLM} are stated since it amounts to bounding the composition of bounded operators $\widetilde{T_1^\omega}^{-1,b}\circ(T_{1R}^\omega + T_{2+}^\omega)$. On the other hand, theorem \ref{thm:iterative} involves a much more technical analysis that will be the main focus of the paper. The proof relies on Propositions \ref{prop:singlescat_inversion_operator} and \ref{prop:rem_mscat}, which give bounds of the error operators $R_1^{\omega,b}$ and $R_2^{\omega,b}$ in some $\mL(L^p(B_{r-D}))$ spaces, respectively. For fixed values of $b$, these norms decay with $\omega$ and so these operators become contractions for large enough $\omega$. 

Both propositions require non-standard estimates on oscillatory integrals (OI) sometimes involving caustic sets (degenerate critical points of the phase function) near which one has to refine the analysis. References \cite{A,AKC,P1,P2} have greatly influenced the approach adopted here. Since proposition \ref{prop:rem_mscat} involves technical analysis as well, we have postponed the stationary phase (SP) part of this proposition to proposition \ref{prop:beta_bound}. Finally, in order to avoid redundancies in the paper, we have formulated an auxiliary lemma \ref{lemma:aux}, which gives, under a non-degeneracy condition, an accurate first-order expansion of a parameter-dependent OI with a simple expression of the remainder. Lemma \ref{lemma:aux} is used to prove propositions \ref{prop:singlescat_inversion_operator} and \ref{prop:beta_bound}.

\section{Forward transport theory, setup for inversion, and proof of theorem \ref{thm:straight_inversion}}

\subsection{Forward transport theory}\label{sec:forward}

We now recall some general results about forward transport that are necessarily in our analysis.
Let $X\subset\Rm^n, n\ge 2$ be an open convex bounded domain with $\mC^1$ boundary $\partial X$ and diameter $\diam>0$. Denote the incoming and outgoing boundaries 
\begin{align}
    \Gamma_\pm = \left\{ (x,v)\in \pa X\times \S^{n-1}\ | \pm \nu_x \cdot v >0 \right\},
    \label{eq:gammapm}
\end{align}
where $\nu_x$ is the outer normal to $\partial X$ at $x\in \partial X$. The time-harmonic transport equation reads
\begin{align}
    \begin{split}
	v\cdot\nabla u + (\sigma(x,v) + i\omega) u &= \int_{\S^{n-1}} k(x,v',v) u(x,v')\ dv', \quad (x,v)\in X\times\S^{n-1}, \\
        u(x,v) &= g(x), \quad (x,v)\in\Gamma_{-},	
    \end{split}    
    \label{eq:thtransport}
\end{align}
where $\omega\ge 0$ and the input function $g$ takes the form $g(x) = \delta(x-x_0), (x_0,x) \in (\partial X)^2$ (call it $g_{x_0}$). By $\delta(x-x_0)$ we mean the delta distribution that satisfies for each smooth function $\psi$ defined at the boundary:
\begin{align}
    \int_{\partial X} \delta(x-x_0) \psi(x) \ d\mu(x) = \psi(x_0), 
    \label{eq:deltaboundary}
\end{align}
where $d\mu(x)$ is the intrinsic measure at the boundary. For $(x,v)\in (X\times\nsphere)\bigcup\Gamma_+\bigcup\Gamma_-$, let $\tau_\pm(x,v)$ be the distance from $x$ to $\partial X$ traveling in the direction of $\pm v$, and $x_\pm(x,v) = x\pm\tau_\pm(x,v)v$ be the boundary point encountered when we travel from $x$ in the direction of $\pm v$.

As it is done in many settings, we integrate \eqref{eq:thtransport} along the direction $v$. We obtain that $u$ is a solution of the following integro-differential equation
\begin{align}
    (I-\mK_\omega) u = J_\omega g, 
    \label{eq:integrodiff}
\end{align}
where we have defined, for $\phi\in L^1(X\times\S^{n-1})$ and $\psi\in L^1(\partial X)$
\begin{align}
    \mK_\omega\phi(x,v)&:=\int_0^{\tau_-(x,v)}e^{-i\omega t}E(x-tv,x)\int_{\S^{n-1}}k(x-tv,v',v)\phi(x-tv,v')dv'dt, \label{eq:Kom} \\
    J_\omega\psi(x,v)&:=e^{-i\omega \tau_-(x,v)}E(x_-(x,v),x)\psi(x_-(x,v)), \quad\mbox{ and } \label{Jom}\\
    E(x,x') &:= \exp\left( -\int_0^{|x-x'|} \sigma(x + \widehat{x'-x}s, \widehat{x'-x})\ ds \right).\label{eq:att}
\end{align}
The operators $\mK_\omega$ and $J_\omega$ are well-defined and continuous operators in $\mL(L^1(X\times\nsphere))$ and $\mL(L^1(\partial X), L^1(X\times\nsphere))$, respectively \cite{dlen6,mokhtar}. 

We now state a theorem recently proved in \cite{BJLM}, which exhibits the asymptotic behavior for high $\omega$ of the solution $u$ of \eqref{eq:thtransport}, and motivates an inversion formula for $k$ ($\lceil\cdot\rceil$ denotes the ceiling part of a real number):

\begin{theorem}\label{thm:dec}
    Assume that $(\sigma,k)\in \mC^{\lceil\frac{n+3}{2}\rceil}(\overline{X}\times\nsphere)\times \mC^{\lceil\frac{n+3}{2}\rceil}_0(X\times\nsphere\times\nsphere)$, and assume hypothesis \ref{hyp:supportk}. Then, for large enough $\omega$, the measurement function $T^\omega$ admits the following singular decomposition
    \begin{align}
	T^\omega(x_0,x_c) = T_0^\omega (x_0,x_c) + T_1^\omega (x_0,x_c) + T_{2+}^\omega (x_0,x_c),
	\label{f1}
    \end{align}
    for a.e. $(x_0,x_c)\in \pa X\times\pa X$, where
    \begin{align}
	T_0^\omega (x_0,x_c) &= {e^{-i\omega|x_c-x_0|}E(x_0,x_c)\over |x_c-x_0|^{n-1}}|\nu_{x_0} \cdot e_0||\nu_{x_c}\cdot e_0|, \label{eq:ballistic}\\
	T_1^\omega(x_0,x_c) &= \widetilde{T_1^\omega}(x_0,x_c) + T_{1R}^\omega(x_0,x_c), \where\nonumber \\
	\widetilde{T_1^\omega}(x_0,x_c) &:= e^{-i\omega d_0} \left( \frac{2\pi}{i d_0\omega} \right)^\frac{n-1}{2} E(x_0,x_c) |\nu_{x_0}\cdot e_0||\nu_{x_c}\cdot e_0| \int_0^{d_0} \frac{k(x_0 + u e_0,e_0,e_0)}{(u (d_0-u))^{n-1\over 2}}\ du, \label{eq:sscat_asym}  \\
	\quad T_{1R}^\omega &\in L^\infty(\partial X\times \partial X)\quad\text{and}\quad \|T_{1R}^\omega\|_\infty \le C_1\ \omega^{-\frac{n+1}{2}}, \label{eq:sscat_remainder} \\
	T_{2+}^\omega &\in L^\infty(\pa X\times\pa X)\quad\text{and }\quad \|T_{2+}^\omega\|_\infty\le\left\lbrace
	\begin{array}{ll}
	    C_2\ \omega^{-1}, & n=2,\\
	    C_2\ \omega^{-2}\ln(\omega), & n=3,\\
	    C_2\ \omega^{-\frac{n+1}{2}}, & n\ge 4,
	\end{array}\label{eq:mscat}\right.
    \end{align}
    where $C_1$ and $C_2$ depend on the $\C^{\lceil\frac{n+3}{2}\rceil}$ and $\C^{\lceil\frac{n+1}{2}\rceil}$ norms of the coefficients $\sigma,k$, respectively.
\end{theorem}

In two dimensions of space, theorem \ref{thm:dec} requires three derivatives on the optical coefficients to ensure that $T_{1R}^\omega$ decays like $\omega^{-\frac{3}{2}}$, and only two derivatives to ensure that $T_{2+}^\omega$ decays like $\omega^{-1}$ in $L^\infty(\partial X\times\partial X)$. Since the total remainder $T_{1R}^\omega + T_{2+}^\omega$ cannot decay faster than $\omega^{-1}$, we choose the $\omega^{-1}$-decay of $T_{1R}^\omega$ and save one derivative on the optical coefficients. Theorem \ref{thm:dec} becomes:
\begin{theorem}\label{thm:dec2d}
    Let $n=2$, and assume that $(\sigma,k)\in \mC^2(\overline{X}\times\nsphere)\times \mC^2_0(X\times\nsphere\times\nsphere)$ and that hypothesis \ref{hyp:supportk} holds. Then the decomposition of theorem \ref{thm:dec} holds with the following bound
    \begin{align}
	\|T_{1R}^\omega\|_\infty + \|T_{2+}^\omega\|_\infty\le C \omega^{-1},
	\label{eq:remainder2d}
    \end{align}
    where the constant $C$ depends on the $\C^2$ norm of the coefficients $(\sigma,k)$. 
\end{theorem}

\begin{proof}
    In the proof of \cite[Lemma 5.2]{BJLM}, do not perform the last integration by parts in the case $m+d$ odd. It trades one derivative of the integrand for a decay factor of $\omega^{-\frac{1}{2}}$.  
\end{proof}


\subsection{The 2D case, when $X=B_r$ and $k(x,v,v') \equiv k(x)\phi(v,v')$}

In this paragraph, we assume that the scattering function admits the expression $k(x,v,v') = k(x) \phi(v, v')$, where $\phi$ is a known phase function. Suppose that the domain $X$ is the centered ball of radius $r>0$ in $\Rm^n$. 
For $x_0$ and $x_c$ located on $\partial X$, we have the properties
\begin{align*}
  (x_c-x_0)\cdot e_0 = |x_c-x_0| \qandq (x_c+x_0)\cdot e_0 = (|x_c|^2-|x_0|^2)|x_c-x_0|^{-1} = 0,
\end{align*}
from which we deduce that $-2x_0\cdot e_0 = |x_c-x_0|$. In this case, 
\begin{align*}
  r^2 - |x_0 + t e_0|^2 = r^2 - |x_0|^2 - 2tx_0\cdot e_0 - t^2 = t|x_c-x_0| - t^2 = |x_c-x_0|t\left( 1-\frac{t}{|x_c-x_0|} \right).
\end{align*}
Hence we can rewrite the first term in equation \eqref{eq:sscat_asym} as
\begin{align}
    \widetilde{T_1^\omega} (x_0,x_c) = e^{-i\omega d_0} \left( \frac{2\pi}{id_0 \omega} \right)^\frac{n-1}{2}E(x_0,x_c)|\nu_{x_0}\cdot e_0||\nu_{x_c}\cdot e_0| \phi(e_0,e_0) \int_0^{d_0}\frac{k(x_0+te_0)} {\left( r^2 - |x_0 + te_0|^2 \right)^\frac{n-1}{2}}\ dt.
  \label{eq:sscat_asym_ball}
\end{align}
In this case, the line integral in \eqref{eq:sscat_asym_ball} is no longer a general weighted integral transform but simply the X-Ray transform of the function $k(x) (r^2-|x|^2)^{\frac{1-n}{2}}$, whose inversion is much more practical. Back to the two-dimensional case, we can now deduce the expression of the function $A^\omega$ defined loosely in \eqref{eq:data_asym}
\begin{align}
    A^\omega(x_0,x_c) =  e^{-i\omega d_0} \left( \frac{2\pi}{i d_0} \right)^\frac{1}{2} E(x_0,x_c)|\nu_{x_0}\cdot e_0||\nu_{x_c}\cdot e_0| \phi(e_0,e_0).
    \label{eq:a_omega}
\end{align}

\subsection{Inversion in two dimensions in parallel geometry}\label{sec:inversion}

\subsubsection{The parallel geometry}

In order to setup an inversion, we must now choose a geometry that is adapted to the X-Ray transform. Namely, we must introduce variables that parameterize lines in space. We choose the parallel geometry since inversion formulas in this case are suitable for our analysis. 

In parallel geometry, we introduce the variables $(s,\theta)\in \mZ:= \Rm\times(0,2\pi)$ and we would like to parameterize the emittor's position $x_0$ and the detector's position $x_c$ in such a way that the line joining $x_0$ and $x_c$ is
\begin{align*}
    L(s,\theta) := \left\{ s\hat\theta^\perp + t\hat\theta, \quad t\in\Rm \right\},
\end{align*}
where we have defined $\hat\theta= (\cos\theta, \sin\theta)$ and $\hat\theta^\perp = (-\sin\theta,\cos\theta)$. With the constraint that $|x_0| = |x_c| = r$, this is done by writing
\begin{align}
    x_0(s,\theta) = s\hat\theta^\perp - \sqrt{r^2-s^2}\hat\theta, \qandq x_c(s,\theta) = s\hat\theta^\perp + \sqrt{r^2-s^2}\hat\theta,
    \label{eq:boundary_radon}
\end{align}
in which case we have $e_0 := \widehat{x_c-x_0} = \hat\theta$ and $d_0(s) := |x_c-x_0| = 2\sqrt{r^2-s^2}$. 

We now denote by $\mZ_D = [-r+D, r-D]\times\S^1$ the support of $\chi(s) \mD^\omega(s,\theta)$. Over $\mZ_D$, we have that $\sqrt{r^2-s^2} \ge \sqrt{rD} > 0$. This guarantees that $x_0$, $x_c$, $d_0(s)$ and $e_0(\theta)$ are well-defined and smooth over $\mZ_D$. 

In the sequel, any function $f(x_0,x_c)$ will be understood as a function defined on $\mZ$ (or on $\mZ_D$ by restriction) with $f(s,\theta) = f(x_0(s,\theta),x_c(s,\theta))$ for $(s,\theta)\in\mZ$. In particular, the function $A^\omega$ \eqref{eq:a_omega} takes the expression
\begin{align}
    A^\omega(s,\theta) = e^{-i\omega 2\sqrt{r^2-s^2}} \left( \frac{2\pi}{i} \right)^\frac{1}{2} e^{-P[\sigma](s,\theta)} \frac{(r^2-s^2)^\frac{3}{4}}{\sqrt{2}r^2} \varphi(\hat\theta,\hat\theta).
    \label{eq:a_omega_radon}
\end{align}
As stated in section \ref{sec:statement}, $(A^\omega)^{-1}$ is uniformly bounded in $\mZ_D$ by
\begin{align}
    \|(A^\omega)^{-1} \|_\infty \le \frac{r^2}{\sqrt{\pi}}  e^{2r\|\sigma\|_\infty} \frac{1}{(rD)^\frac{3}{4}} \left( \min_{\theta} |\varphi(\hat\theta,\hat\theta)| \right)^{-1} <\infty,
    \label{eq:a_omega_radon_bound}
\end{align}
independently of $\omega$. $(A^\omega)^{-1}$ is also in $L^1(\mZ_D)$, hence in every $L^p(\mZ_D)$ by interpolation.

\subsubsection{Radon Transform and Filtered-Backprojection operators} 
In parallel geometry, for functions $f$ and $g$ respectively defined on $\Rm^2$ and $\mZ$, we define the following Radon transform $P$ and backprojection $P^\sharp$ by:
\begin{align}
    P [f](s,\theta) &:= \int_{L(\theta,s)} f = \int_{\Rm} f(s\hat\theta^\perp + t\hat\theta)\ dt,\ (s,\theta)\in\mZ, \label{eq:radon} \\
    P^\sharp [g](x) &:= \int_{\Sone} g(x\cdot\hat\theta^\perp, \theta)\ d\theta,\ x\in\Rm^2. \label{eq:backproj}
\end{align}
Note that $P$ and $P^\sharp$ are adjoint to each other in the spaces $L^2(\Rm^2)$ and $L^2(\mZ)$ since
\begin{align*}
    \int_\mZ g(s,\theta) P[f](s,\theta)\ ds\ d\theta = \int_{\Rm^2} P^\sharp[g](x) f(x)\ dx.
\end{align*}
In a 2D setting, inversion of the Radon transform can be done by means of the following formula
\begin{align}
    f = \frac{1}{4\pi} P^\sharp \circ \pdr{}{s} P[f].
    \label{eq:IRT}
\end{align}
This formula is however practically never used, because it requires differentiating the possibly noisy data as it tries to reconstruct $f$ at arbitrarily small scales and this generates instabilities in reconstructions.  Rather, we introduce a regularized version of \eqref{eq:IRT} based on the following observation that
\begin{align}
    P^\sharp [w_b \stackrel{s}{\star} P[f]] = W_b \stackrel{x}{\star} f = \frac{1}{2\pi} \int_{\Rm^2} e^{ix\cdot\xi} \hat\Phi \left( \frac{|\xi|}{b} \right) \hat f(\xi)\ d\xi,
    \label{eq:FBP}
\end{align}
where $\stackrel{x}{\star}$ and $\stackrel{s}{\star}$ respectively denote 2D and 1D convolutions in the $x$ and $s$ variables and $W_b = P^\sharp w_b = \mF^{-1}[\hat \Phi]$ is an approximation of identity ($\mF^{-1}$ denotes the inverse Fourier transform in 2D). Following \cite{natt} in two dimensions of space, we fix a function $\hat\Phi$ first, which is a cutoff function in frequency supported inside $[0,1]$. From $\hat\Phi$, we define the filter $w_b$ by 
\begin{align}
    w_b(u) = \frac{1}{8\pi^2}\int_\Rm e^{i\sigma u} |\sigma| \hat\Phi\left( \frac{|\sigma|}{b} \right)\ d\sigma, 
    \label{eq:wb}
\end{align}
where $b$ plays the role of a bandwidth. Notice that the function $w_b$ in \eqref{eq:wb} naturally satisfies the identities
\begin{align}
    w_b(u) = b^2 w_1(bu), \qandq w_b^{(n)}(u) = b^{2+n} w_1^{(n)}(bu),\ n\ge 0.
    \label{eq:wbidentity}
\end{align}
For $g$ defined on $\mZ$, let us define the filtered-backprojection (FBP) operator $P^{-1,b}$ by
\begin{align}
    P^{-1,b}[g] = P^\sharp[ w_b \stackrel{s}{\star} g],
    \label{eq:reg_FBP}
\end{align}
and note the following straightforward estimate when $g\in L^\infty(\mZ)$, useful for section \ref{sec:proof1}:
\begin{align}
    \|P^{-1,b}[g]\|_{L^\infty(B_r)} \le 2\pi b\|w_1\|_{L^1} \|g\|_\infty.
    \label{eq:reg_FBP_estimate}
\end{align}
Using convolution properties of the operators \eqref{eq:radon} and \eqref{eq:backproj}, we can also express the FBP operator as follows \cite{natt}
\begin{align}
    P^{-1,b} = \frac{1}{4\pi} I^{-1,b}\circ P^\sharp, \where\quad  I^{-1,b}f := \mF^{-1} \left[ \hat\Phi\left( \frac{|\xi|}{b} \right) |\xi| \mF[f]  \right].
    \label{eq:reg_FBP2}
\end{align}

\subsection{Proof of theorem \ref{thm:straight_inversion} }\label{sec:proof1}

We recall that 
\begin{align*}
    R^{\omega,b}[k] = \widetilde{T_1^\omega}^{-1,b} [T_{1R}^\omega[k] + T_{2+}^\omega[k]] = \sqrt{\omega} P^{-1,b} \left[(A^\omega)^{-1}\chi[T_{1R}^\omega[k] + T_{2+}^\omega[k]]\right]
\end{align*}
From theorem \ref{thm:dec2d}, we have 
\begin{align}
    \|T_{1R}^\omega[k] + T_{2+}^\omega[k]\|_{L^\infty((\partial X)^2)} \le C\omega^{-1},
    \label{eq:rem_bound}
\end{align}
where $C$ depends on $\|\sigma\|_{\mC^2}$ and $\|k\|_{\mC^2}$. Reparameterizing $T_{1R}^\omega[k] + T_{2+}^\omega[k]$ into the cylinder variables $(s,\theta)$ yields a same bound as \eqref{eq:rem_bound}, but in the space $L^\infty(\mZ)$. Hence, using estimates \eqref{eq:a_omega_radon_bound} and \eqref{eq:reg_FBP_estimate}, we get
\begin{align}
    \|R^{\omega,b}[k]\|_{L^\infty(B_{r-2D})} &\le \sqrt{\omega} 2\pi b\|w_1\|_{L^1} \|(A^\omega)^{-1}\|_\infty \|T_{1R}^\omega[k] + T_{2+}^\omega[k]\|_\infty
    \label{eq:thm1_bound}
\end{align}
Combining inequalities \eqref{eq:rem_bound} and \eqref{eq:thm1_bound} concludes the proof.

\section{Proof of theorem \ref{thm:iterative}}\label{sec:proof2}

\subsection{Auxiliary lemma} \label{sec:aux_lemma}
In the sequel, we use the following lemma twice, which is an exact first-order stationary phase expansion of an 1D oscillatory integral with parameters. This particular case, where the second derivative of the phase function is uniformly bounded away from zero along one direction, is useful for reducing the dimensionality of a given OI while keeping under control the possible singularities (i.e. degenerate critical points) of the phase function.

Let us consider the following integral for $\lambda\in\Rm^n$, $n\ge 0$,
\begin{align}
    I(\lambda) = \int_{\Rm} e^{i\omega\varphi(x,\lambda)} f(x,\lambda)\ dx,
    \label{eq:I_lambda}
\end{align}
where $f$ is assumed to be compactly supported in some set of the form $[-\Delta_x, \Delta_x]\times B$, $B$ compact in $\Rm^n$, and $\varphi$ is smooth on the support of $f$. Assume further the existence of $0 < \Phi_m \le \Phi_M <\infty$ such that 
\begin{align}
    \Phi_m\le |\partial_{xx}^2\varphi(x,\lambda)|\le \Phi_M, \quad (x,\lambda) \in \supp\ f.
    \label{eq:condition_phi}
\end{align}
The lower bound in \eqref{eq:condition_phi} ensures that for every $\lambda\in B$, there exists a unique $X(\lambda)$ such that $\partial_x\varphi(X(\lambda),\lambda) = 0$, since $\partial_x\varphi(\cdot,\lambda)$ is strictly monotonous. By the implicit functions theorem, $X$ is a smooth function of $\lambda$. Condition \eqref{eq:condition_phi} allows us to make the phase globally quadratic in the $x$ variable (which otherwise would require the Morse lemma in a neighborhood of the critical point, the size of which we could not control). We thus write
\begin{align}
    \varphi(x,\lambda) &= S(\lambda) + \sigma K(\lambda) \frac{\eta(x,\lambda)^2}{2}, \where \label{eq:phi_quadratic} \\[2mm]
    S(\lambda) &:= \varphi(X(\lambda),\lambda), \quad K(\lambda):= |\partial_{xx}^2 \varphi(X(\lambda),\lambda)|, \qandq \sigma = \sgn{\partial_{xx}^2\varphi} \nonumber
\end{align}
is clearly constant on $\supp\ f$. The lemma is stated as follows:
\begin{lemma}
    Let a one-dimensional parameter-dependent OI be given by $I(\lambda)$ in \eqref{eq:I_lambda}, where $\varphi$ satisfies the condition \eqref{eq:condition_phi}. Then $I(\lambda)$ admits the following decomposition $I = I_0 + I_r$, with
    \begin{align}
	I_0(\lambda) &:= e^{i\sigma\frac{\pi}{4}} \left( \frac{2\pi}{\omega} \right)^{\frac{1}{2}} e^{i\omega S(\lambda)} \frac{f(X(\lambda),\lambda)}{K(\lambda)} \label{eq:I_f_expr} \\
	|I_r(\lambda)| &\le \frac{C_r}{\omega} \min \left( \frac{1}{\sqrt{\omega}} \sum_{i=0}^3 \left\| \frac{\partial^i f}{\partial x^i}(\cdot,\lambda) \right\|_{L^1(\Rm)}, \sum_{i=0}^2 \left\| \frac{\partial^i f}{\partial x^i}(\cdot,\lambda) \right\|_{L^1(\Rm)} \right), \quad\lambda\in B. \label{eq:bound_I_g}
    \end{align}
    \label{lemma:aux}
\end{lemma}
\begin{proof} 
    Starting from expression \eqref{eq:phi_quadratic} and using Taylor expansions, we have
   \begin{align*}
       \eta &= \sgn{x-X(\lambda)} \left[ \frac{2}{K(\lambda)} (\varphi(x,\lambda)-S(\lambda)) \right]^\frac{1}{2} \\
       &= (x-X(\lambda)) \left[ \frac{2}{K(\lambda)} \int_0^1 (1-t)  |\partial_{xx}^2\varphi|(tx + (1-t)X(\lambda))\ dt \right]^\frac{1}{2}.
   \end{align*}
   From this relation, we can show that the support of $f$, now a function of $(\eta,\lambda)$, is included in $[-H_0,H_0]\times B$, where 
   \begin{align}
       H_0 := \frac{2\Delta_x}{\sqrt{\Phi_m}} \sqrt{2\Phi_M}.
       \label{eq:H0}
   \end{align}
   The corresponding jacobian is given by 
   \begin{align*}
       \pdr{\eta}{x} = \frac{1}{\sqrt{2 K(\lambda)}} \frac{\int_0^1 |\partial_{xx}^2\varphi|(tx + (1-t)X(\lambda))\ dt}{\sqrt{\int_{0}^1 (1-t)|\partial_{xx}^2\varphi|(tx + (1-t)X(\lambda))\ dt}},
   \end{align*}
   and is clearly well-defined everywhere by virtue of the bounds \eqref{eq:condition_phi} and smooth. Changing variable $x\to \eta$ in \eqref{eq:I_lambda}, we get
   \begin{align}
       I(\lambda) = e^{i\omega S(\lambda)} \int_{\Rm} e^{i\omega\sigma K(\lambda) \frac{\eta^2}{2}} f(x(\eta,\lambda),\lambda) \pdr{x}{\eta}(\eta,\lambda)\ d\eta.
       \label{eq:I_lambda2}
   \end{align}
   We now write a Taylor expansion of $f \partial_\eta x$ with integral remainder in the $\eta$ variable (note that $\partial_\eta x=1$ at $\eta=0$ and that $x(0,\lambda)=X(\lambda)$):
   \begin{align}
       f(x(\eta,\lambda),\lambda)\pdr{x}{\eta} = f(X(\lambda),\lambda) + \eta\int_0^1 \pdr{}{\eta}\left[ f(x(\cdot,\lambda),\lambda) \pdr{x}{\eta}(\cdot,\lambda)  \right](t\eta)\ dt.
       \label{eq:Taylor1}
   \end{align}
   In order to make the second term of the RHS of \eqref{eq:Taylor1} compactly supported, we write the following: let $\chi$ is a smooth function of $\eta$, supported in $[-2H_0,2H_0]$ and identically one on $[-H_0,H_0]$, then $\chi f = f$. Multiplying \eqref{eq:Taylor1} by $\chi(\eta)$, we rewrite it as 
   \begin{align}
       f(x(\eta,\lambda),\lambda)\pdr{x}{\eta} &= f(X(\lambda),\lambda) + \eta g(\eta,\lambda) + h(\eta,\lambda),\quad (\eta,\lambda)\in\Rm\times B,  \where \label{eq:Taylor2} \\
       g(\eta,\lambda) &:= \chi(\eta) \int_0^1 \pdr{}{\eta}\left[ f(x(\cdot,\lambda),\lambda) \pdr{x}{\eta}(\cdot,\lambda) \right](t\eta)\ dt, \label{eq:Taylor_remainder} \\
       h(\eta,\lambda) &:= f(X(\lambda),\lambda) (1-\chi(\eta)). \label{eq:rd_rem}
   \end{align}
   Plugging \eqref{eq:Taylor2} into \eqref{eq:I_lambda2} yields the decomposition $I = I_0 + I_{r1} + I_{r2}$, where   
   \begin{align}
       I_0(\lambda) &:= e^{i\omega S(\lambda)} f(X(\lambda),\lambda) \int_\Rm e^{i\omega\sigma K(\lambda)\frac{\eta^2}{2}}\ d\eta = e^{i\sigma\frac{\pi}{4}} \left( \frac{2\pi}{\omega} \right)^{\frac{1}{2}} e^{i\omega S(\lambda)} \frac{f(X(\lambda),\lambda)}{K(\lambda)}, \label{eq:I_f}
   \end{align}
   where we used the formula $\int_\Rm e^{i\sigma\beta\frac{\eta^2}{2}}\ d\eta = e^{i\sigma\frac{\pi}{4}} \sqrt{\frac{2\pi}{\beta}}$ that holds for $\sigma = \pm 1$ and $\beta>0$. The remainder integrals are expressed as
   \begin{align}
       I_{r1}(\lambda) &:= e^{i\omega S(\lambda)} \int_{\Rm} \eta e^{i\omega\sigma K(\lambda) \frac{\eta^2}{2}} g(\eta,\lambda)\ d\eta = -\sigma\frac{e^{i\omega S(\lambda)}}{i\omega K(\lambda)} \int_{\Rm} e^{i\omega\sigma K(\lambda) \frac{\eta^2}{2}} \pdr{g}{\eta}(\eta,\lambda)\ d\eta, \label{eq:I_g}\\
       I_{r2}(\lambda) &:= f(X(\lambda),\lambda) \int_{\Rm} e^{i\omega \sigma K(\lambda) \frac{\eta^2}{2}}(1-\chi(\eta))\ d\eta, \label{eq:I_h}
   \end{align}
   where we used the formula $\eta e^{i\omega\sigma K(\lambda) \frac{\eta^2}{2}} = \frac{1}{i\omega\sigma K(\lambda)}\partial_\eta e^{i\omega\sigma K(\lambda) \frac{\eta^2}{2}}$ to integrate by parts in $I_{r1}(\lambda)$. We integrate by parts once again in \eqref{eq:I_g} and obtain this time
   \begin{align}
       I_g(\lambda) = \sigma \frac{e^{i\omega S(\lambda)}}{i (\omega K(\lambda))^\frac{3}{2}} \int_{\Rm} F_\sigma (\eta \sqrt{\omega K(\lambda)}) \pdrr{g}{\eta} (\eta,\lambda)\ d\eta,
       \label{eq:I_g2}
   \end{align}
   where the function $F_\sigma(u) := \int_{-\infty}^u e^{i\sigma \frac{t^2}{2}}\ dt$ is related to the Fresnel integral and is known to be uniformly bounded in $\Rm$. 
   From the two expressions of $I_{r1}(\lambda)$ in \eqref{eq:I_g2} and in the last right-hand side of \eqref{eq:I_g}, the expression of $g$ in \eqref{eq:Taylor_remainder} and the fact that $\pdr{^n x}{\eta^n}$ is bounded for $n \le 4$, we obtain that $I_{r1}$ satisfies a bound of the form \eqref{eq:bound_I_g}. 
   Moreover, the remaining integral $I_{r2}$ \eqref{eq:I_h} is rapidly decaying because $1-\chi(\eta)$ is supported away from $\eta=0$, so it can be made comparable to (if not smaller than) $I_{r1}$. Indeed, we have for every $p\ge 0$   
   \begin{align}
       I_h(\lambda) = \left( \frac{i}{\sigma K(\lambda) \omega}\right)^p f(X(\lambda),\lambda) \int_{\Rm} e^{i\omega\sigma K(\lambda)\frac{\eta^2}{2}} \left[ \pdr{}{\eta} \circ \frac{1}{\eta} \right]^p [1-\chi](\eta)\ d\eta,
       \label{eq:I_h2}
   \end{align}
   where we have used $p$ integrations by parts. From equation \eqref{eq:I_h2} with $p=1$, we obtain a bound of the form 
   \begin{align*}
       |I_{r2}(\lambda)| \le C |I_{r1}(\lambda)|, \quad \lambda\in B,
   \end{align*}
   where the constant depends only on the features of the function $\chi$. As a conclusion, setting $I_r := I_{r1} + I_{r2}$, it is clear that $I_r$ satisfies bound \eqref{eq:bound_I_g}. The proof is complete.  
\end{proof}

\subsection{Main outline and proof}
In order to prove theorem \ref{thm:iterative}, we must bound the operators $R_1^{\omega,b}$ in \eqref{eq:rem1} and $R_2^{\omega,b}$ in \eqref{eq:rem2} and make them contraction operators in some spaces to be defined. This is achieved by doing stationary-phase analyses of the compositions of operators instead of bounding each operator separately as was done when proving theorem \ref{thm:straight_inversion}.

The proposition below bounds the error operator $R_1^{\omega,b}$ in \eqref{eq:rem1} and shows how close the operator $\widetilde{T_1^\omega}^{-1,b} T_1^\omega$  comes to providing $[\rho k]_b$.
\begin{proposition}\label{prop:singlescat_inversion_operator}
    The operator $R_1^{\omega,b}[k]$ is bounded in $\mL(L^2(B_{r-2D}))$ and $\mL(L^\infty(B_{r-2D}))$ and satisfies the following bounds for $\omega\ge 1$
    \begin{align}
	\|R_1^{\omega,b} \|_{\mL(L^2(B_{r-2D}))} &\le \frac{C_1}{\omega} (b + b^4) \label{eq:l2_bound} \\
	\|R_1^{\omega,b} \|_{\mL(L^\infty(B_{r-2D}))} &\le \frac{C_1}{\omega} (b^2 + b^5) \label{eq:linf_bound}.
    \end{align}
\end{proposition}
Let us now introduce the operator $L$ (first introduced in \cite{BLM}) defined as:
\begin{align}
    Lu(x) = \int_X \frac{u(y)E(x,y)}{|x-y|}\ dy.
    \label{eq:L_op}
\end{align}
It is proved in \cite{BLM} for instance that $L$ is well-defined in $\mL(L^p), \,p\in [1,\infty]$. Its operator norm in each of these spaces, denoted by $\|L\|_p$, satisfies the estimate
\begin{align*}
    \|L\|_p \le \olL := \sup_{x\in X} \int_X \frac{E(x,y)}{|x-y|}\ dy \le 2\pi\diam.
\end{align*}

The next proposition controls the second error operator $R_2^{\omega,b}$ \eqref{eq:rem2}, provided that the scattering coefficient is uniformly bounded by $(\|\phi\|_\infty \|L\|_\infty)^{-1}$.
\begin{proposition}\label{prop:rem_mscat}
    Let $0 < K_0< (\|\phi\|_\infty \|L\|_\infty)^{-1}$ and $k, \tilde k\in B_{K_0}(L^\infty(B_{r-2D}))$. Then, there exists a constant $C_2$ that depends on $\|\sigma\|_{\C^2}$, $K_0$ and $\diam$ such that, we have 
    \begin{align}
	\left\|R_2^{\omega,b}[k] - R_2^{\omega,b}[\tilde k]\right\|_\infty \le C_2 \frac{b^3}{\omega^{\frac{1}{2}}}\log\left( \frac{\omega}{b} \right) \|k-\tilde k\|_\infty.
	\label{eq:rem_mscat}
    \end{align}
\end{proposition}

\begin{remark}
    The boundedness conditions for $k,\tilde k$ in proposition \ref{prop:rem_mscat} are used to ensure that the scattering series converges. They are sufficient but not necessary. A less stringent condition would be to impose that the operators $L_{\sigma,k}$ and $L_{\sigma,\tilde k}$ have operator norms in $\L(L^p(X\times\Sone))$ bounded by  $1$, where we have defined
    \begin{align*}
	L_{\sigma,k} f(x,v) = \int_X k(x') \frac{E(x',x)}{|x-x'|} \phi(\widehat{x-x'},v) f(x', \widehat{x-x'})\ dx, \quad (x,v)\in X\times\Sone. 
\end{align*}
For the sake of clarity, however, we do not consider such a proof here. 
\end{remark}

The proofs of propositions \ref{prop:singlescat_inversion_operator} and \ref{prop:rem_mscat} above are given consecutively in the next two subsections. We now prove theorem \ref{thm:iterative}.
 
\begin{proof}[Proof of theorem \ref{thm:iterative}]
    Let us fix a reconstruction bandwidth $b$ and $0 < K_0< (\|\phi\|_\infty \|L\|_\infty)^{-1}$. Applying the inversion to the data yields the equality
    \begin{align}
	\widetilde{T_1^\omega}^{-1,b} \mD^\omega = [k\rho]_b + R_1^{\omega,b}[k\rho] + R_2^{\omega,b}[k\rho]:= [k\rho]_b + R^{\omega,b}[k\rho].
	\label{eq:krhob}
    \end{align}
    The introduction of $\rho$ into the argument of the remainder operators may change the constants $C_1$, $C_2$ of \eqref{eq:linf_bound} and \eqref{eq:rem_mscat} but causes no trouble since $\rho$ is bounded from above and away from zero on $B_{r-2D}$. The bounds \eqref{eq:linf_bound} and \eqref{eq:rem_mscat} decrease to zero as $\omega\to\infty$ so that there exists $\omega_0$ such that for $\omega\ge\omega_0$, we have 
    \begin{align*}
	\frac{C_1}{\omega} (b^2 + b^5) +  C_2 \frac{b^3}{\omega^{\frac{1}{2}}}\log\left( \frac{\omega}{b} \right) = c_1 < 1.
    \end{align*}
    Thus, when $\omega\ge\omega_0$, the operator $R^{\omega,b} = R_1^{\omega,b} + R_2^{\omega,b}$ is a contraction on $B_{K_0}(L^\infty(B_{r-2D}))$. We must now show that the mapping $\tilde k\mapsto \widetilde{T_1^\omega}^{-1,b} \mD^\omega - R^{\omega,b}(\tilde k)$ is a contraction in $B_{K_1}(L^\infty(B_{r-2D}))$ for some $0<K_1\le K_0$. In order to do so, we write, using \eqref{eq:krhob}, 
    \begin{align*}
	\widetilde{T_1^\omega}^{-1,b} \mD^\omega - R^{\omega,b}[\tilde k] = [k\rho]_b + R^{\omega,b}[k\rho] - R^{\omega,b}[\tilde k].
    \end{align*}
    For a given $x\in B_{r-2D}$, we have
    \begin{align*}
	|[k\rho]_b(x)| &= \left|\int_{\Rm^2} W_b(y-x) [k\rho](y)\ dy\right| \le \|W_b\|_1 \|k\rho\|_\infty, \where \\
	W_b (x) &= \int_{\Rm^2} e^{ix\cdot\xi} \hat\Phi\left( \frac{|\xi|}{b} \right)\ d\xi = b^2 W_1 (bx), \quad\text{ thus }\quad \|W_b\|_{L^1} = \|W_1\|_{L^1},
    \end{align*}
    and so we obtain the bound
    \begin{align*}
	\|\widetilde{T_1^\omega}^{-1,b} \mD^\omega - R^{\omega,b}(\tilde k)\|_\infty \le (\|W_1\|_{L^1} + c_1) \| k\rho \|_\infty + c_1 \|\tilde k\|_\infty.
    \end{align*}
    Therefore, setting $K_1:=\frac{K_0(1-c_1)}{(\|W_1\|_1 + c_1)}$ and assuming that $\| k\rho \|_\infty\le K_1$, this ensures that $\|\tilde k\|_\infty\le K_0$ implies $\|\widetilde{T_1^\omega}^{-1,b} \mD^\omega - R^{\omega,b}(\tilde k)\|_\infty\le K_0$. Hence, whenever $\| k\rho \|_\infty\le K_1$, the iterative scheme \eqref{eq:iterated_scheme} converges to a limit point $k^\star\in B_{K_0}(L^\infty(B_{r-2D}))$ that satisfies
    \begin{align}
	\widetilde{T_1^\omega}^{-1,b} \mD^\omega = k^\star + R^{\omega,b}[k^\star].
	\label{eq:kstar}
    \end{align}
    Subtracting \eqref{eq:kstar} from \eqref{eq:krhob} yields
    \begin{align*}
	k^\star - [k\rho]_b = R^{\omega,b}[k\rho] - R^{\omega,b}[k^\star].
    \end{align*}
    Thus, we have
    \begin{align*}
	\|k^\star - [k\rho]_b\|_\infty &\le \|R^{\omega,b}[k\rho] - R^{\omega,b}[k^\star]\|_\infty \le c_1 \|k\rho - k^\star\|_\infty \\
	&\le c_1 \|k\rho - [k\rho]_b\|_\infty + c_1\|k^\star - [k\rho]_b\|_\infty, \qhenceq\\
	\|k^\star - [k\rho]_b\|_\infty &\le \frac{c_1}{1-c_1} \|k\rho - [k\rho]_b\|_\infty.
    \end{align*}
    Theorem \ref{thm:iterative} is proved. 
\end{proof}

\subsection{Proof of proposition \ref{prop:singlescat_inversion_operator}}

The operator $R_1^{\omega,b}$ is linear in $k$, and we now exhibit its Schwartz kernel. We recall that
\begin{align*}
    T_1^\omega[k](s,\theta) = \int_{\Rm^2} e^{-i\omega(|x-x_0|+|x-x_c|)} \phi(\widehat{x-x_0},\widehat{x_c-x}) E(x_0,x,x_c) \frac{|\widehat{x-x_0}\cdot\nu_{x_0}||\widehat{x-x_c}\cdot\nu_{x_c}|}{|x-x_0||x-x_c|\rho(x)} [\rho k](x)\ dx.
\end{align*}
We now apply to it the inverse operator \eqref{eq:inversion_operator}, using the expression \eqref{eq:reg_FBP2} of the FBP operator. We first have
\begin{align}
    &(A^\omega)^{-1} \chi T_1^\omega[k](s,\theta) = \sqrt{\frac{i}{2\pi}} \int_{\Rm^2} e^{i\omega\varphi_1(s,\theta,x)} c(s,\theta,x) [\rho k](x)\ dx,\where \nonumber \\
    \varphi_1(s,\theta,x) &:= |x_c-x_0| - |x-x_0| - |x-x_c|, \label{eq:phi1}\\
    c(s,\theta,x) &:= \chi(s) \chi(|x|) \frac{\varphi(\widehat{x-x_0},\widehat{x_c-x})}{\varphi(\hat\theta,\hat\theta)} \frac{E(x_0,x,x_c)}{E(x_0,x_c)} \frac{|\widehat{x-x_0}\cdot\nu_{x_0}||\widehat{x-x_c}\cdot\nu_{x_c}|}{|x-x_0||x-x_c|\rho(x)} \frac{\sqrt{2}r^2}{(r^2-s^2)^{\frac{3}{4}}}. \label{eq:c_function}
\end{align}
Applying now a backprojection and exchanging the integral signs, we write
\begin{align}
    \sqrt{\omega} P^\sharp[(A^\omega)^{-1} \chi T_1^\omega[k]](y) 
    &= \mF^{-1} \left[ a(y,\xi,\omega) \mF[\rho k] \right](y),\where \nonumber \\
    a(y,\xi,\omega) &:= \sqrt{\frac{i\omega}{2\pi}} \int_{\Rm^2} \int_\Sone e^{i(x-y)\cdot\xi} e^{i\omega\varphi_1(\yperp,\theta,x)} c(\yperp,\theta,x)\ d\theta\ dx. \label{eq:a_symbol}
\end{align}
Finally applying the operator $I^{-1,b}$ in \eqref{eq:reg_FBP2}, the operator $\widetilde{T_1^\omega}^{-1,b} T_1^\omega[k]$ can be expressed as
\begin{align}
    \widetilde{T_1^\omega}^{-1,b} T_1^\omega[k] (y) = \frac{1}{16\pi^3} \int_{\Rm^2} e^{iy\cdot\xi} \hat\Phi\left( \frac{|\xi|}{b} \right) |\xi| a(y,\xi,\omega) \mF[\rho k](\xi)\ d\xi,
    \label{eq:prelim_single}
\end{align}
with $a$ defined in \eqref{eq:a_symbol}.

\begin{proof}[Proof of proposition \ref{prop:singlescat_inversion_operator}]
    We now analyze the symbol $a(y,\xi,\omega)$ in \eqref{eq:a_symbol} by studying the behavior of the phase function $\varphi_1$ \eqref{eq:phi1} in terms of $x$. For fixed $y\in B_{r-2D}$ and $\theta\in\Sone$, we have $\nabla_x\varphi_1 = 0$ whenever
    \begin{align*}
	\widehat{x-x_0} + \widehat{x-x_c} = 0,
    \end{align*}
    which happens exactly when $x \in [x_0,x_c]$. It is thus convenient to use the change of variable $x = y + u\hat\theta + v\hat\theta^\perp$ for $(u,v)\in\Rm^2$, and the symbol $a$ becomes
    \begin{align}
	a(y,\xi,\omega) &= \sqrt{\frac{i\omega}{2\pi}} \int_\Sone \int_\Rm e^{iu\xi\cdot\hat\theta} I(y,\xi,u,\theta)\ du\ d\theta, \where \label{eq:symbol2}\\
	I(y,\xi,u,\theta) &= \int_{\Rm} e^{i\omega\varphi_1(\yperp,\theta,u,v)} f_1(v,y,u,\theta,\xi)\ dv \label{eq:Ia} \\
	f_1(v,u,y,\theta) &:= e^{iv\xiperp} c(\yperp,\theta,u,v). \label{eq:f1a}
    \end{align}
    We are now in the setting of section \ref{sec:aux_lemma}, where the amplitude function $f_1$ is supported in 
    \begin{align}
	\supp\ f_1 &\subset\left\{ (v,u,y,\theta)\in\Rm^2\times B_{r-D}\times\Sone \text{ s.t. } x(u,v,y,\theta) \in B_{r-D} \text{ and } |\yperp|\le r-D \right\} \\
	&\subset [-\diam,\diam]\times[-\diam,\diam]\times B_{r-D}\times \Sone.
	\label{eq:domain}
    \end{align}
    On the other hand, the phase function, expressed as  
    \begin{align*}
	\varphi_1 = |x_0-x_c| - \sqrt{v^2 + |Px-x_0|^2} - \sqrt{v^2 + |Px-x_c|^2}, \qquad (Px := y+u\hat\theta)
    \end{align*}
    is proved in appendix \ref{app:phi1} to satisfy 
    \begin{align}
	-\infty<-\Phi_{1,M} \le \pdrr{\varphi_1}{v} \le -\Phi_{1,m} < 0, \quad\text{with}\quad \Phi_{1,m} := \frac{D}{4r^2}, \quad \Phi_{1,M}:= \frac{2}{D},
	\label{eq:d2phi1dv}
    \end{align}
    uniformly on $\supp\ f_1$. Moreover, since $\varphi_1$ is clearly even in $v$, the unique $v$ that satisfies $\partial_v\varphi_1 = 0$ is $v=0$ regardless of all other variables. Hence, section \ref{sec:aux_lemma} gives us the expression of a rigorous decomposition for the integral $I$ in \eqref{eq:Ia}. Note that in this case, we have
    \begin{align*}
	S_1(u,y,\theta) := \varphi_1(\yperp, \theta, u, 0) = 0, &\quad K_1(u,y,\theta) := \partial_{vv}^2\varphi_1|_{v=0} = -d_0 \rho(y+u\hat\theta)^2. \\
	f_1(0,u,y,\theta) &= \sqrt{d_0} \rho(y+u\hat\theta).
    \end{align*}
    Thus the integral $I$ in \eqref{eq:Ia} decomposes into $I_0 + I_r$, where, following formula \eqref{eq:I_f},
    \begin{align*}
	I_0(u,y,\theta) = e^{-i\frac{\pi}{4}} \left( \frac{2\pi}{\omega} \right)^\frac{1}{2} e^{i\omega S_1(u,y,\theta)} \frac{f_1(0,u,y,\theta)}{K_1(u,y,\theta)} = \left( \frac{2\pi}{i\omega} \right)^\frac{1}{2}.
    \end{align*}
    The remainder term $I_r$ satisfies the following bound, following \eqref{eq:bound_I_g}:
    \begin{align}
	|I_r(u,y,\theta)| \le \frac{C_r}{\omega^\frac{3}{2}} \sum_{i=0}^3 \left\|\frac{\partial^i f_1}{\partial v^i} (\cdot,u,v,\theta) \right\|_{L^1(\Rm)}.
	\label{eq:Ig1_bound}
    \end{align}
    From expression \eqref{eq:f1a}, it is clear that for $0\le i\le 3$,
    \begin{align}
	\left\|\frac{\partial^i f_1}{\partial v^i} \right\|_{L^1(\Rm)} \le \sum_{j=0}^i \binom{i}{j} |\xi|^j \left\|\frac{\partial^j c}{\partial v^j} \right\|_{L^1(\Rm)},
	\label{eq:f_1a_c}
    \end{align}
    which in turn implies the bound 
    \begin{align}
	|I_r(u,\theta,y)| \le \frac{C}{\omega^\frac{3}{2}} (1+ |\xi|^3), \quad (u,\theta) \in\supp\ c.
	\label{eq:Ig1_xi}
    \end{align}
    Plugging the decomposition $I = I_0 + I_r$ into expression \eqref{eq:symbol2} yields the following decomposition $a = a_0 + a_r$, where
    \begin{align*}
	a_0(y,\xi,\omega) &:= \sqrt{\frac{i\omega}{2\pi}} \int_{\Sone}\int_\Rm e^{iu\xi\cdot\hat\theta} I_0(u,y,\theta)\ du\ d\theta = \int_{\Sone}\int_{\Rm} e^{iu\xi\cdot\hat\theta} \ du\ d\theta \stackrel{x=u\hat\theta}{=} 2 \int_{\Rm^2} \frac{e^{i x\cdot\xi}}{|x|}\ dx = \frac{4\pi}{|\xi|}. \\
	a_r(y,\xi,\omega) &:= \sqrt{\frac{i\omega}{2\pi}} \int_{\Sone}\int_\Rm e^{iu\xi\cdot\hat\theta} I_r(u,y,\theta)\ du\ d\theta. 
    \end{align*}
    By virtue of bound \eqref{eq:Ig1_xi}, we get the following bound on $a_r$:
    \begin{align}
	|a_r(y,\xi,\omega)| \le \frac{C}{\omega} (1 + |\xi|^3).
	\label{eq:a2_bound}
    \end{align}
    Plugging $a=a_0+a_r$ into expression \eqref{eq:prelim_single} shows that $\widetilde{T_1^\omega}^{-1} T_{1}^\omega[k](y)$ is the sum of two terms, the first one being 
    \begin{align*}
	\frac{1}{16\pi^3} \int_{\Rm^2} e^{iy\cdot\xi} \hat\Phi\left( \frac{|\xi|}{b} \right) |\xi| a_0(y,\xi,\omega) \mF[\rho k](\xi)\ d\xi = \frac{1}{4\pi^2} \int_{\Rm^2} e^{iy\cdot\xi} \hat\Phi\left( \frac{|\xi|}{b} \right) \mF[\rho k](\xi)\ d\xi = [\rho k]_b (y),
    \end{align*}
    and the second one being the error operator
    \begin{align*}
	R_1^{\omega,b}[k](y) = \frac{1}{16\pi^3} \int_{\Rm^2} e^{iy\cdot\xi} \hat\Phi\left( \frac{|\xi|}{b} \right) |\xi| a_r(y,\xi,\omega) \mF[\rho k](\xi)\ d\xi.
    \end{align*}
    Taking $L^2$ norms, we obtain the first bound
    \begin{align*}
	\int_{\Rm^2} |R_1^{\omega,b}[k](y)|^2\ dy &= C_1 \int_{\Rm^2} \hat\Phi\left( \frac{|\xi|}{b} \right)^2 |\xi|^2 |a_r(y,\xi,\omega)|^2 |\mF[\rho k](\xi)|^2\ d\xi \\
	&\le C_1 \frac{C^2}{\omega^2} \int_{\Rm^2} \hat\Phi\left( \frac{|\xi|}{b} \right)^2 (|\xi| + |\xi|^4)^2 |\mF[\rho k](\xi)|^2\ d\xi \\
	&\le C_1 \frac{C^2}{\omega^2} (b + b^4)^2 \int_{\Rm^2} |\mF[\rho k](\xi)|^2\ d\xi,
    \end{align*}
    where we used that $0\le \hat\Phi\le 1$ and that on the support of $\hat\Phi\left( \frac{|\xi|}{b} \right)$, we have $|\xi|\le b$. Using Parseval on the latter right-hand side, this yields the $\mL(L^2)$ continuity inequality 
    \begin{align*}
	\|R_1^{\omega,b}[k]\|_{L^2} \le \sqrt{C_1} \frac{C}{\omega} (b + b^4) \|\rho k\|_{L^2},
    \end{align*}
    and hence the bound in \eqref{eq:l2_bound} since $\rho$ is uniformly bounded on $B_{r-D}$. 

    In order to get an $\mL(L^\infty)$ bound, we have the following estimate
    \begin{align*}
	|R_1^{\omega,b}[k](y)| &\le \frac{C}{16\pi^3 \omega } \int_{\Rm^2} \hat\Phi\left( \frac{|\xi|}{b} \right) (|\xi|+|\xi|^4) |\mF[\rho k](\xi)|\ d\xi \\
	&\le \frac{C}{16\pi^3\omega} (b^2 + b^5) \|\rho k\|_{L^2} \|\hat\Phi\|_{L^2},
    \end{align*}
    where we used the following estimate based on the Cauchy-Schwarz inequality:
    \begin{align*}
	\int_{\Rm^2} |\xi|^q \hat\Phi\left( \frac{|\xi|}{b} \right) |\mF[g](\xi)|\ d\xi \le \|g\|_{L^2}\|\hat\Phi\|_{L^2} b^{q+1}, \quad q\ge 0,
    \end{align*} 
    and hence the bound \eqref{eq:linf_bound}. This completes the proof.
\end{proof}

\subsection{Proof of proposition \ref{prop:rem_mscat}}
For $\alpha$, $\beta$ formal subscripts, denote $v_{\alpha,\beta}:= \widehat{x_\beta-x_\alpha}$. For $m\ge 2$, denoting $\bfx_m = (x_1,\dots,x_m)$ and $d\bfx_m = dx_1\dots dx_m$, the multiple scattering of order $m$ is given by
\begin{align}
    \begin{split}
	T_m^\omega[k](s,\theta) &= \int_{X^m} E^{\omega}(x_m,x_c) \prod_{j=1}^m \frac{k(x_j)\phi(v_{j-1,j}, v_{j,j+1}) E^{\omega}(x_{j-1},x_j)}{|x_{j}-x_{j-1}|} \frac{|\nu_{x_0}\cdot v_{0,1}||\nu_{x_c}\cdot v_{m,c}|}{|x_m-x_c|}\ d\bfx_m.
    \end{split}
    \label{eq:tm}
\end{align}

We now multiply pointwise by $(A^\omega)^{-1}(s,\theta) \chi(s)$ and apply the FBP operator using expression \eqref{eq:reg_FBP}. The operator $\widetilde{T_1^\omega}^{-1,b} T_{m}^\omega$ can be written in the compact way
\begin{align}
    \widetilde{T_1^\omega}^{-1,b} T_{m}^\omega[k] (y) = \sqrt{\omega} \int_{X^m} f_m^\omega[k] (\bfx_m) \beta^\omega (y,x_1,x_2,x_{m-1},x_m)\ d\bfx_m,
    \label{eq:mscat_compact}
\end{align}
where we have defined 
\begin{align*}
    f_m^\omega[k](\bfx_m) = k(x_1) k(x_m) \frac{E^\omega(x_{m-1},x_m)}{|x_m-x_{m-1}|}\prod_{j=2}^{m-1} \frac{k(x_j) E^\omega(x_{j-1},x_j) \phi(v_{j-1,j}, v_{j,j+1})}{|x_j-x_{j-1}|},
\end{align*}
(with the last product replaced by $1$ when $m=2$), as well as 
\begin{align}
    \beta^\omega (y,x_1,x_2,x_{m-1},x_m) &= \int_{\Sone} \int_\Rm e^{i\omega\varphi_2(s,\theta,x_1,x_m)} w_b(\yperp-s) \alpha^\omega(s,\theta,x_1,x_2,x_{m-1},x_m) \chi(s)\ ds\ d\theta, \label{eq:betam} \\
    \alpha^\omega(s,\theta,x_1,x_2,x_{m-1},x_m) &:= (A^\omega)^{-1}(s,\theta) \frac{E(x_0,x_1)E(x_m,x_c)}{|x_1-x_0||x_m-x_c|} \nonumber \\
    &\times |\nu_{x_0}\cdot v_{0,1}||\nu_{x_c}\cdot v_{m,c}| \phi(v_{0,1}, v_{1,2})\phi(v_{m-1,m}, v_{m,c}), \label{eq:alpha_m} \\
    \varphi_2(s,\theta,x_1,x_m) &:= |x_0(s,\theta)-x_c(s,\theta)|-|x_0(s,\theta)-x_1|-|x_c(s,\theta)-x_m|. \label{eq:phi_mscat}
\end{align}

Since we do not assume any regularity on $k$, we cannot use stationary phase with respect to a variable on which $k$ depends. This is why stationary phase arguments will be used to bound $\beta^\omega$ but not the $f_m^\omega[k]$'s.
This requires estimates on $\beta^\omega$ in order to bound the operator $\widetilde{T_1^\omega}^{-1,b} T_{m}^\omega$ using H\" older inequalities. We will use the following:
\begin{proposition}\label{prop:beta_bound}
    The function $\beta^\omega$ admits the decomposition $\beta^\omega = \beta_1^\omega + \beta_2^\omega$, where the functions $\beta_1^\omega$ and $\beta_2^\omega$ satisfy the following bounds for every $(y,x_1,x_2,x_{m-1},x_m)\in (B_{r-D})^{5}$ such that $x_1\ne x_m$:
    \begin{align}
	|\beta_1^\omega(y,x_1,x_2,x_{m-1},x_m)| &\le \overline{\beta_1^\omega}(x_1,x_m) := C_{\beta,1} \frac{b^3}{\omega} \frac{1}{\sqrt{|x_1-x_m|}}, \label{eq:bound_beta1} \\
	|\beta_2^\omega(y,x_1,x_2,x_{m-1},x_m)| &\le \overline{\beta_2^\omega}(x_1,x_m) := C_{\beta,2} \frac{b^\frac{5}{2}}{\omega} \frac{1}{|x_1-x_m|}. \label{eq:bound_beta2}
    \end{align}
\end{proposition}
The proof is rather technical and is postponed to section \ref{sec:beta_bound}. We are now in a position to bound the operator $R_2^{\omega,b}$ \eqref{eq:rem2}. 

\begin{proof}[Proof of proposition \ref{prop:rem_mscat}]
    We start from equality \eqref{eq:mscat_compact} and obtain estimates for $f_m^\omega$. Let $k, \tilde k\in L^\infty(B_{r-2D})$. Then using the following property that holds for real numbers $a_1,\dots,a_m,b_1,\dots,b_m$
    \begin{align}
	a_1\cdots a_m - b_1\cdots b_m = \sum_{j=1}^m b_1\cdots b_{j-1} (a_j-b_j) a_{j+1}\cdots a_m,
	\label{eq:ajbj_property}
    \end{align}
    with $a_j = k(x_j), b_j= \tilde k(x_j), \ j=1\dots m$, and the fact that $\max(\|k\|_\infty, \|\tilde k\|_\infty)\le K_0$, we have the following estimate
    \begin{align}
	|f_m^\omega[k](\bfx)-f_m^\omega[\tilde k](\bfx)| \le m K_0^{m-1} \|\phi\|^{m-2} \left[ \prod_{j=2}^m \frac{E(x_{j-1},x_j)}{|x_j-x_{j-1}|} \right] \|k-\tilde k\|_\infty.
	\label{eq:bound_fmk}
    \end{align}
    Using estimates \eqref{eq:bound_beta1} and \eqref{eq:bound_beta2} together with \eqref{eq:bound_fmk}, we have that for $k,\tilde k \in B_{K_0}(L^\infty(B_{r-2D}))$ and $m\ge 2$,
    \begin{align}
	\left\|\widetilde{T_1^\omega}^{-1} T_{m}^\omega[k] - \widetilde{T_1^\omega}^{-1} T_{m}^\omega[\tilde k]\right\|_\infty &\le \sqrt{\omega} K_0 m (K_0 \|\phi\|)^{m-2} \|k-\tilde k\|_\infty [C_{m,1} + C_{m,2}], \label{eq:bound_tmdiff} \\
	\text{where}\quad C_{m,\alpha} &:= \int_{X^2} \overline{\beta_\alpha^\omega} (x_1,x_m) p_m (x_1,x_m)\ dx_1\ dx_m, \quad \alpha=1,2, \label{eq:Cm_expr} \\
	p_m(x_1,x_m) &:= \int_{X^{m-2}} \prod_{j=2}^m \frac{E(x_{j-1}, x_j)}{|x_j-x_{j-1}|}\ dx_2\dots\ dx_{m-1} = L^{m-2}\left[ \frac{E(x_1,\cdot)}{|x_1-\cdot|} \right](x_m). \nonumber
    \end{align}
    In order to bound the $C_{m,\alpha}$'s, we need estimates on $p_m$. We have the following estimates (obtained by bounding the attenuation terms by $1$ and performing polar changes of variable inside the integrals): 
    \begin{align*}
	p_2(x_1, x_2) &\le |x_2-x_1|^{-1}, \text{a.e. } (x_1, x_2), \\
	p_3(x_1, x_3) &\le C_3 |\log(|x_3-x_1|)|, \text{a.e. } (x_1, x_3), \\
	p_m(x_1, x_m) &\le C_4 \|L\|_\infty^{m-4}, \text{a.e. } (x_1,x_m),\ m\ge 4.
    \end{align*}
    Hence by virtue of the bounds \eqref{eq:bound_beta1} and \eqref{eq:bound_beta2}, the only term that requires attention is $C_{2,2}$. The integrals defining all the other $C_{m,\alpha}$ are well-defined. Addressing $C_{2,2}$, we change variables $x_m = x_1 + \rho\hat\alpha$ and split the integral in $\rho$ into two intervals $[0,\varepsilon]$ and $[\varepsilon,\diam]$ and use different bounds for $\overline{\beta_2^\omega}$ as follows
    \begin{align*}
	C_{2,2} &\le C_{\beta,2}\frac{b^2}{\sqrt{\omega}}\left[ \int_X \int_\Sone \int_0^\varepsilon\ d\rho\ d\alpha\ dx_1 + \left( \frac{b}{\omega} \right)^\frac{1}{2} \int_X \int_\Sone \int_\varepsilon^\diam\ \frac{d\rho}{\rho}\ d\alpha\ dx_1 \right] \\
	&\le C_{\beta,2} 2\pi|X| \frac{b^2}{\sqrt{\omega}} \left[ \varepsilon + \left( \frac{b}{\omega} \right)^\frac{1}{2} \left( \log\diam - \log\varepsilon \right) \right] \le C'_{2,2} \frac{b^{\frac{5}{2}}}{\omega}\log\frac{\omega}{b},
    \end{align*}
    if we choose $\varepsilon = \left( \frac{b}{\omega} \right)^\frac{1}{2}$. For the other terms $C_{m,\alpha}$, we obtain bounds of the form 
    \begin{align*}
	C_{2,1} &\le C_{2,1}' \frac{b^3}{\omega}, \qquad C_{3,\alpha} \le C_{3,\alpha}' \frac{b^3}{\omega}, \quad \alpha=1,2 \\
	C_{m,\alpha} &\le C_{4} \frac{b^3}{\omega} \|L\|_\infty^{m-4}, \quad \alpha=1,2,\ m\ge 4.
    \end{align*}
    Hence for $m=2$ we have 
    \begin{align}
	\left\|\widetilde{T_1^\omega}^{-1} T_{2}^\omega[k] - \widetilde{T_1^\omega}^{-1} T_{2}^\omega[\tilde k]\right\|_\infty &\le 2K_0 \left[ C'_{2,1} \frac{b^3}{\sqrt{\omega}} + C'_{2,2} \frac{b^{\frac{5}{2}}}{\sqrt{\omega}} \log\left( \frac{\omega}{b} \right) \right] \|k-\tilde k\|_\infty, 
	\label{eq:bound_m2}
    \end{align}
    while for $m=3$ we have
    \begin{align}
	\left\|\widetilde{T_1^\omega}^{-1} T_{3}^\omega[k] - \widetilde{T_1^\omega}^{-1} T_{3}^\omega[\tilde k]\right\|_\infty &\le 3K_0^2 \|\phi\| [C'_{3,1} + C'_{3,2}] \frac{b^3}{\sqrt{\omega}} \|k-\tilde k\|_\infty.
	\label{eq:bound_m3}
    \end{align}
    Summing inequalities in \eqref{eq:bound_tmdiff} for $m\ge 4$, we get 
    \begin{align}
	\sum_{m=4}^\infty \left\|\widetilde{T_1^\omega}^{-1} T_{m}^\omega[k] - \widetilde{T_1^\omega}^{-1} T_{m}^\omega[\tilde k]\right\|_\infty &\le 2 C_4 \frac{K_0}{\|L\|_\infty^2} \frac{b^3}{\sqrt{\omega}} \|k-\tilde k\|_\infty \sum_{m=4}^\infty m (K_0 \|\phi\|\|L\|_\infty)^{m-2} \nonumber \\
	&\le 2 C_4 \frac{K_0}{\|L\|_\infty^2}\frac{1}{(1-K_0 \|\phi\|\|L\|_\infty)^2} \frac{b^3}{\sqrt{\omega}} \|k-\tilde k\|_\infty,
	\label{eq:bound_m4}
    \end{align}
    where we used the following inequality for $|x|\le 1$:
    \begin{align*}
	\sum_{m=4}^\infty m x^{m-2} \le \frac{1}{(1-x)^2}.
    \end{align*}
    Finally using the fact that
    \begin{align*}
	\left\|R_2^{\omega,b}[k] - R_2^{\omega,b}[\tilde k]\right\|_\infty &\le \sum_{m=2}^\infty \left\|\widetilde{T_1^\omega}^{-1,b} T_m^\omega[k] - \widetilde{T_1^\omega}^{-1,b} T_m^\omega[\tilde k]\right\|_\infty,
    \end{align*}
    we obtain the desired result using bounds \eqref{eq:bound_m2}, \eqref{eq:bound_m3} and \eqref{eq:bound_m4}.
\end{proof}

\subsection{Proof of proposition \ref{prop:beta_bound}} \label{sec:beta_bound}

The function $\beta^\omega$ is an OI in $(s,\theta)$ parameterized by the points $(y,x_1,x_2,x_{m-1},x_m)\in (B_{r-2D})^{5}$, which we now analyze using stationary phase techniques. The only parameters that matter for the analysis are those the phase function $\varphi_2$ in \eqref{eq:phi_mscat} depends upon, that is $x_1$ and $x_m$, and all estimates are uniform in $(y,x_2,x_{m-1})$. Critical points of $\varphi_2$ become degenerate as $|x_1-x_m|\to 0$ and so a careful analysis is necessary to control this degenerate behavior and capture the leading term as $\omega\to\infty$.

Let us analyze the phase function $\varphi_2$ in \eqref{eq:phi_mscat} first and in particular which couples $(s,\theta)$ satisfy $(\partial_s\varphi_2, \partial_\theta\varphi_2)=(0,0)$. Using identities from section \ref{sec:appendix}, the gradient of $\varphi_2$ in the $(s,\theta)$ variables can be written under the form:
\begin{align}
    \partial_s\varphi_2 = \frac{r}{\sqrt{r^2-s^2}} ( f_0 - f_c ), &\qandq \partial_\theta \varphi_2 = r (f_0 + f_c), \where\nonumber \\
    f_0 := (\widehat{x_1-x_0}- \hat\theta)\cdot\widehat{\partial_s x_0}, &\quad f_c := (\widehat{x_c-x_m}-\hat\theta)\cdot\widehat{\partial_s x_c}. \label{eq:f0fc}
\end{align}
Now, since $\sqrt{r^2-s^2}$ is never zero on $\mZ_D$, the system of equations $(\partial_s\varphi_2, \partial_\theta\varphi_2) = (0,0)$ can be recast as
\begin{align*}
    \sqrt{r^2-s^2} \partial_s\varphi_2 \pm \partial_\theta\varphi_2 = 0, 
\end{align*}
which is equivalent to $f_0=f_c=0$. This in turn implies $\widehat{x_1-x_0} = \widehat{x_c-x_m} = \hat\theta$. Hence the critical points of $\varphi_2$ are exactly those for which the points $x_0, x_1, x_m, x_c$ are aligned, that is:
\begin{description}
    \item[if $x_1\ne x_m$,] two distinct points $(s, \theta)\in\{ (x_1\cdot\widehat{\theta_{1m}}^\perp, \theta_{1m}), (x_1\cdot\widehat{\theta_{1m}+\pi}^\perp, \theta_{1m}+\pi) \}$, where we defined 
	\begin{align}
	    \theta_{1m} := \arg(x_m-x_1).
	    \label{eq:theta1m}
	\end{align}
	At these points, we have $\varphi_2 = |x_1-x_m|$ and $-|x_1-x_m|$, respectively.
    \item[if $x_1=x_m$,] a curve in $\mZ_D$ of equation $\{ s = x_1\cdot\hat\theta^\perp \}$, along which $\varphi_2=0$. 
\end{description}
The second case describes the caustic set in the space of parameters, i.e. the parameter values for which the stationary points are degenerate. Since $\varphi_2$ is constant along this curve, it is clear that these stationary points are degenerate, see \cite{P1,P2}. This degenerescence can be observed by noticing that the determinant of the Hessian at the two critical points in the first case is proportional to $|x_1-x_m|$ and hence goes to zero as $|x_1-x_1|\to 0$. 
The degenerate critical points are not of type $A_k$ (as defined in \cite{A}) for any $k\ge 2$, since the phase is locally constant along one direction as one passes through a degenerate critical point. However, because at any of these critical points, $\text{rank } H_{\varphi_2} \ge 1$, the approach of \cite[Section 3]{P2} which we are following here remains efficient in order to control the degenerescence near the caustic set. 

Following this approach, we first need to find a direction along which the second derivative of $\varphi$ is uniformly bounded away from zero. In our case, it is proved in appendix \ref{subapp:dphi2} that for any $(s,\theta)\in\mZ_D$,
\begin{align}
    -\infty< -\Phi_{2,M} \le \pdrr{\varphi_2}{s} (s,\theta) \le -\Phi_{2,m} < 0, \quad\text{with}\quad \Phi_{2,m} = \frac{2D}{r^2} \qandq \Phi_{2,M} = \frac{8r}{D^2}.
    \label{eq:bound_phi_ss}
\end{align}

As a result, for every $\theta\in \S^1$, the function $\partial_s \varphi_2(\cdot,\theta)$ is a strictly decreasing function. Moreover, it is easy to check that $\partial_s\varphi_2(x_1\cdot\hat\theta^\perp,\theta)$ and $\partial_s\varphi_2(x_m\cdot\hat\theta^\perp,\theta)$ have opposite signs. This ensures that for every $\theta\in\Sone$, the equation $\partial_s\varphi_2(s,\theta)=0$ has a unique solution called $\sigma(\theta)$, which satisfies the following bounds
\begin{align}
    \min (x_1\cdot\hat\theta^\perp, x_m\cdot\hat\theta^\perp) \le \sigma(\theta) \le \max(x_1\cdot\hat\theta^\perp, x_m\cdot\hat\theta^\perp).
    \label{eq:bound_sigma}
\end{align}
Except for the simple case where $x_1\cdot\hat\theta^\perp = x_m\cdot\hat\theta^\perp = \sigma(\theta)$ (i.e. $x_0, x_1, x_m, x_c$ are aligned), it appears that the computation of $\sigma(\theta)$, which requires solving the equation $f_0 = f_c$ (with $f_0$, $f_c$ defined in \eqref{eq:f0fc}) leads to quite complicated expressions. It is however important to note that, because $\sigma(\theta)$ satisfies the following ODE (derived from the implicit functions theorem)
\begin{align}
    \frac{d\sigma}{d\theta} = -\frac{\partial^2_{s\theta}\varphi (\sigma(\theta), \theta)}{\partial^2_{ss}\varphi (\sigma(\theta), \theta)},
    \label{eq:ODE_sigma}
\end{align}
and because $\varphi_2$ is smooth with respect to all of its arguments, then $\sigma(\theta)$ is a $\mC^\infty$ function of $\theta$.
For the sequel, we now define the functions
\begin{align}
    S_2(\theta) := \varphi_2(\sigma(\theta),\theta), \quad K_2(\theta) := |\partial_{ss}^2\varphi_2(\sigma(\theta), \theta)| = -\partial_{ss}^2\varphi_2(\sigma(\theta), \theta).
    \label{eq:S_function}
\end{align}
Let us now assume $x_1\ne x_m$. First notice that, using the chain rule as well as relation \eqref{eq:ODE_sigma}, the first two derivatives of $S_2$ \eqref{eq:S_function} are given by
\begin{align*}
    S_2'(\theta) = \partial_\theta\varphi_2(\sigma(\theta), \theta), \qandq S_2''(\theta) = - \frac{\det H_{\varphi_2} (\sigma(\theta), \theta)}{K_2(\theta)}.
\end{align*}
The critical points of $S_2$ (i.e. the zeros of $S_2'$) are $\theta_{1m}$ and $\theta_{1m} +\pi$, at which we have 
\begin{align}
    S_2''(\theta_{1m}) = - \frac{|x_1-x_m||x_0-x_c|}{|x_1-x_0| + |x_c-x_m|}, \qandq S''(\theta_{1m}+\pi) = \frac{|x_1-x_m||x_0-x_c|}{|x_1-x_0| + |x_c-x_m|}.
    \label{eq:det_values}
\end{align}

\begin{proof}[Proof of proposition \ref{prop:beta_bound}]
    Let us denote by 
    \begin{align}
	\beta^\omega(y, x_1,x_2,x_{m-1},x_m) &= \int_\Sone I_2(\theta)\ d\theta, \where\quad I_2(\theta) = \int_{\Rm} e^{i\omega\varphi_2} f(s,\theta) \chi(s)\ ds, \label{eq:I2} \\
	f(s,\theta) &:= \alpha^\omega (s,\theta,x_1,x_2,x_{m-1},x_m) w_b(\yperp-s). \label{eq:fm}
    \end{align}
    The dependency on parameters is made implicit to improve readability. Using the identities \eqref{eq:wbidentity} satisfied by $w_b$ and the fact that $\alpha^\omega$ is smooth, we have that for every $0\le i\le 2$, there exists a constant $C_{f,i}$ such that
    \begin{align}
	\left\| \frac{\partial^i f}{\partial s^i}(\cdot,\theta) \right\|_{L^1(\Rm)} \le C_{f,i} b(1+b^i), \quad \theta\in\Sone. 
	\label{eq:f_norms}
    \end{align}
    Let us focus on the OI $I_2(\theta)$ in \eqref{eq:I2} first. Again, since we have the property \eqref{eq:bound_phi_ss}, we can use the results of section \ref{sec:aux_lemma}. We can decompose $I_2 = I_{2,0} + I_{2,r}$, where following \eqref{eq:I_f}, 
    \begin{align}
	I_{2,0}(\theta) &= e^{-i\frac{\pi}{4}} \left( \frac{2\pi}{\omega} \right)^\frac{1}{2} \frac{f(\sigma(\theta), \theta)}{K_2(\theta)^\frac{1}{2}}. \label{eq:I2f} \\
	|I_{2,r}|(\theta) &\le \frac{C}{\omega} \sum_{i=0}^2 \left\|\frac{\partial^i f}{\partial s^i}(\cdot,\theta)\right\|_{L^1(\Rm)} \le \frac{C_r}{\omega} b (1 + b^2). \label{eq:I2gI2h}
    \end{align}
    Integrating now $I_{2,0}(\theta)$ with respect to $\theta$, we get an oscillatory integral $J_{2,0}:= \int_{\S^1} I_{2,0}(\theta)\ d\theta$ with phase function $S_2(\theta)$. We recall that the critical set of $S_2$ is $\Theta_c = \{\theta_{1m}, \theta_{1m}+\pi\}$. We now claim the following
    \begin{lemma} \label{lemma:S}
	There exist constants $\frac{\pi}{2}> \delta_0 >0$, $C_1>0$ and $C_2>0$ independent of $(x_1,x_m)$ such that the following inequalities hold:
	\begin{align}
	    |S_2'(\theta)| &\ge C_1 |x_1-x_m| \text{ for } \min_{\theta_c\in\Theta_c} |\theta-\theta_c| \ge \frac{\delta_0}{2}, \label{eq:bound_Sp} \\
	    |S_2''(\theta)| &\ge C_2 |x_1-x_m| \text{ for } \min_{\theta_c\in\Theta_c} |\theta-\theta_c| \le \delta_0. \label{eq:bound_Spp} 
	\end{align} 
    \end{lemma}
    Estimate \eqref{eq:bound_Sp} will help us bound the integral $J_{2,0}$ away from the critical points, while estimate \eqref{eq:bound_Sp} will help us bound the contribution to $J_{2,0}$ in a neighborhood of the critical points. Lemma \ref{lemma:S} is relegated to appendix \ref{app:lemma}.
    Let $\delta_0$, $C_1$ and $C_2$ be given by lemma \ref{lemma:S}. Let $\rho_1(\theta)$ be a smooth function such that $\rho_1(\theta) = 0$ if $|\theta|\ge \delta_0$ and $\rho_1(\theta) = 1$ if $|\theta|\le \frac{\delta_0}{2}$. We now write the following partition of unity
    \begin{align*}
	1 = \rho_c(\theta) + \tilde \rho(\theta), \where\quad \rho_c(\theta) := \sum_{\theta_c\in\Theta_c} \rho_1(\theta-\theta_c),
    \end{align*}
    thus $\tilde\rho(\theta) = 0$ if $\min_{\theta_c\in\Theta_c} |\theta-\theta_c| \ge \frac{\delta_0}{2}$. Using this partition of unity, the integral $J_{2,0}$ can be written as
    \begin{align*}
	J_{2,0} &= e^{-i\frac{\pi}{4}} \left( \frac{2\pi}{\omega} \right)^\frac{1}{2} \left[ J_{2,0,1} + J_{2,0,2} \right], \where \\
	J_{2,0,1} &:= \int_{\S^1} e^{i\omega S_2(\theta)} \frac{f(\sigma(\theta),\theta)}{K_2(\theta)^\frac{1}{2}} \rho_c(\theta)\ d\theta, \qandq J_{2,0,2} := \int_{\S^1} e^{i\omega S_2(\theta)} \frac{f(\sigma(\theta),\theta)}{K_2(\theta)^\frac{1}{2}} \tilde\rho(\theta)\ d\theta.
    \end{align*}
    The integrand in  $J_{2,0,1}$ vanishes for $\min_{\theta_c\in\Theta_c} |\theta-\theta_c| \le \delta_0$ and so we can use estimate \eqref{eq:bound_Spp} together with \cite[Corollary 2]{AKC} to bound $J_{2,0,1}$:
    \begin{align}
	|J_{2,0,1}| \le \frac{12}{\sqrt{\omega C_2 |x_1-x_m|}} \left\| \pdr{}{\theta} \left[  \frac{f(\sigma(\theta),\theta)}{K_2(\theta)^\frac{1}{2}} \rho_c(\theta) \right] \right\|_{L^1(\Sone)} \le C_{201} \frac{b^3}{\sqrt{\omega|x_1-x_m|}}.
	\label{eq:bound_211}
    \end{align}
    Treating $J_{2,0,2}$ now, a first bound can be obtained easily 
    \begin{align}
	|J_{2,0,2}| \le \left\| \frac{f(\sigma(\theta),\theta)}{K_2(\theta)^\frac{1}{2}} \tilde\rho (\theta) \right\|_{L^1(\Sone)} \le C_{202} b^2.
	\label{eq:bound_212_1}
    \end{align}
    Then, using the fact that the integrand in $J_{2,0,2}$ vanishes for $\theta$ such that $\min_{\theta_c\in\Theta_c} |\theta-\theta_c|\ge \frac{\delta_0}{2}$, we can use the bound \eqref{eq:bound_Sp} to integrate by parts:
    \begin{align*}
	J_{2,0,2} = \frac{1}{i\omega} \int_{\S^1} \frac{d}{d\theta}\left[ e^{i\omega S_2(\theta)} \right] \frac{f(\sigma(\theta)) \tilde\rho(\theta)}{K_2(\theta)^\frac{1}{2} S_2'(\theta)}\ d\theta = \frac{-1}{i\omega} \int_{\S^1} e^{i\omega S_2(\theta)} \frac{d}{d\theta} \left[ \frac{f(\sigma(\theta)) \tilde\rho(\theta)}{K_2(\theta)^\frac{1}{2} S_2'(\theta)} \right]\ d\theta,
    \end{align*}
    whence the bound
    \begin{align}
	|J_{2,0,2}| \le \frac{1}{\omega C_1^2 |x_1-x_m|^2} \left\| S_2'(\theta)^2 \frac{d}{d\theta} \left[ \frac{f(\sigma(\theta)) \tilde\rho(\theta)}{K_2(\theta)^\frac{1}{2} S_2'(\theta)} \right] \right\|_{L^1(\Sone)} \le C_{202}' \frac{b^3}{\omega |x_1-x_m|^2}. 
	\label{eq:bound_212_2}
    \end{align}
    Taking the geometric average of bounds \eqref{eq:bound_212_1} and \eqref{eq:bound_212_2}, we get
    \begin{align}
	|J_{2,0,2}| \le (C_{202} C_{202}')^{\frac{1}{2}} \frac{b^{\frac{5}{2}}}{\omega^{\frac{1}{2}} |x_1-x_m|}.
	\label{eq:bound_212_3}
    \end{align}
    To conclude, if we define $\beta_1^\omega := e^{-i\frac{\pi}{4}}\left( \frac{2\pi}{\omega} \right)^{\frac{1}{2}} J_{2,0,1} + \int_{\Sone} I_{2,r}(\theta)\ d\theta$ and $\beta_2^\omega := e^{-i\frac{\pi}{4}}\left( \frac{2\pi}{\omega} \right)^{\frac{1}{2}} J_{2,0,2}$, we obtain the desired decomposition. Estimate \eqref{eq:bound_beta1} is obtained by using inequalities \eqref{eq:bound_211} and setting $C_{\beta,1}:= \sqrt{2\pi} C_{201} + 2\pi C_r$, while estimate \eqref{eq:bound_beta2} is obtained by using inequalities \eqref{eq:bound_212_1} and \eqref{eq:bound_212_3} and setting $C_{\beta,2}:= \sqrt{2\pi} \max(C_{202}, (C_{202}C_{202}')^\frac{1}{2})$.

\end{proof}

\appendix

\section{Formulas for proposition \ref{prop:singlescat_inversion_operator}}\label{app:phi1}
From equality 
\begin{align*}
    \varphi_1 = |x_c-x_0| - \sqrt{v^2 + |Px-x_0|^2} - \sqrt{v^2 + |Px-x_c|^2},
\end{align*}
we get that 
\begin{align*}
    \pdrr{\varphi_1}{v} = -\left[\frac{|Px-x_0|^2}{|x-x_0|^3} + \frac{|Px-x_c|^2}{|x-x_c|^3} \right].
\end{align*}
Thus we have 
\begin{align*}
    \frac{|Px-x_0|^2}{|x-x_0|^3} + \frac{|Px-x_c|^2}{|x-x_c|^3} &\ge \frac{1}{8r^3} (|Px-x_0|^2+|Px-x_c|^2) \ge \frac{(|Px-x_0|+|Px-x_c|)^2}{16r^3} \\
    &\ge \frac{|x_0-x_c|^2}{16r^3} = \frac{r^2-s^2}{4r^3} \ge \frac{D}{4r^2}, 
\end{align*}
and hence the lower bound on $|\partial_{vv}^2\varphi_1|$. The upper bound is given by
\begin{align*}
    \frac{|Px-x_0|^2}{|x-x_0|^3} + \frac{|Px-x_c|^2}{|x-x_c|^3} \le \frac{1}{|x-x_0|} + \frac{1}{|x-x_c|} \le \frac{2}{D}. 
\end{align*}

\section{Formulas for proposition \ref{prop:beta_bound}}\label{sec:appendix}

\subsection{Identities involving the boundary points $x_0$, $x_c$}
Concerning the boundary points in the Radon geometry, we have
\begin{align*}
    \pdr{x_0}{s} = \frac{r}{\sqrt{r^2-s^2}}\left( \sqrt{1-\frac{s^2}{r^2}} \hat\theta^\perp + \frac{s}{r}\hat\theta \right), &\quad \pdr{x_c}{s} = \frac{r}{\sqrt{r^2-s^2}}\left( \sqrt{1-\frac{s^2}{r^2}} \hat\theta^\perp - \frac{s}{r}\hat\theta \right), \\
    \pdr{x_0}{\theta} = -\sqrt{r^2-s^2} \pdr{x_0}{s}, &\quad \pdr{x_c}{\theta} = \sqrt{r^2-s^2} \pdr{x_0}{s}. 
\end{align*}
The second partial derivatives satisfy the relations:
\begin{align*}
    \pdrr{x_0}{\theta} = -x_0, &\quad \pdrr{x_0}{s} = \frac{r^2}{(r^2-s^2)^{\frac{3}{2}}}\ \hat\theta, \quad \pdrt{x_0}{s}{\theta} = \frac{1}{\sqrt{r^2-s^2}}\ x_0, \\
    \pdrr{x_c}{\theta} = -x_c, &\quad \pdrr{x_c}{s} = \frac{-r^2}{(r^2-s^2)^{\frac{3}{2}}}\ \hat\theta, \quad \pdrt{x_c}{s}{\theta} = \frac{-1}{\sqrt{r^2-s^2}}\ x_c.
\end{align*}

\subsection{Properties of $\varphi_2$} \label{subapp:dphi2}
Derivatives of the phase function: we have 
\begin{align*}
    \varphi_2 &= |x_0-x_c| - |x_0-x_1| - |x_m-x_c|, \\
    \partial_s\varphi_2 &= \frac{-r}{\sqrt{r^2-s^2}}\left( 2\frac{s}{r} + \widehat{x_0-x_1}\cdot\widehat{\partial_s x_0} + \widehat{x_c-x_m}\cdot\widehat{\partial_s x_c} \right) \\
    &= \frac{r}{\sqrt{r^2-s^2}}\left( (\widehat{x_1-x_0}- \hat\theta)\cdot\widehat{\partial_s x_0} - (\widehat{x_c-x_m}-\hat\theta)\cdot\widehat{\partial_s x_c} \right) \\
    \partial_\theta\varphi_2 &= r\left( \widehat{x_0-x_1}\cdot\widehat{\partial_\theta x_0} + \widehat{x_c-x_m}\cdot\widehat{\partial_\theta x_c} \right) \\
    &= r\left(  (\widehat{x_1-x_0}- \hat\theta)\cdot\widehat{\partial_s x_0} + (\widehat{x_c-x_m}-\hat\theta)\cdot\widehat{\partial_s x_c} \right).
\end{align*}
In order to compute the second derivatives, we use the following identity, where $x$ depends on $(a,b)$ and $y$ is a constant vector:
\begin{align*}
    \pdrt{}{a}{b} |x-y| = \widehat{x-y}\cdot\pdrt{x}{a}{b} + \frac{1}{|x-y|}\left( \pdr{x}{a}\cdot\pdr{x}{b} - \left( \widehat{x-y}\cdot\pdr{x}{a} \right)\left( \widehat{x-y}\cdot\pdr{x}{b} \right) \right).
\end{align*} 
We obtain
\begin{align*}
    \pdrr{\varphi_2}{s} &= \frac{- r^2}{(r^2-s^2)^\frac{3}{2}}\left[ 2+(\widehat{x_0-x_1}-\widehat{x_c-x_m})\cdot\hat\theta + \sqrt{r^2-s^2}( f_0 + f_c) \right] \\
    \pdrt{\varphi_2}{s}{\theta} &= \frac{1}{\sqrt{r^2-s^2}} \left[ -\widehat{x_0-x_1}\cdot x_0 + \widehat{x_c-x_m}\cdot x_c + r^2 ( f_0 - f_c ) \right] \\
    \pdrr{\varphi_2}{\theta} &= \widehat{x_0-x_1}\cdot x_0 + \widehat{x_c-x_m}\cdot x_c - r^2 ( f_0 + f_c ),
\end{align*}
where we have defined 
\begin{align*}
    f_0 := \frac{1-(\widehat{x_0-x_1}\cdot\widehat{\partial_s x_0})^2}{|x_1-x_0|} \qandq f_c := \frac{1-(\widehat{x_c-x_m}\cdot\widehat{\partial_s x_c})^2}{|x_c-x_m|}.
\end{align*}
Since $x_0$ (resp. $x_c$) is orthogonal to $\partial_s x_0$ (resp. $\partial_s x_c$), we can rewrite $f_0$ and $f_c$ as 
\begin{align*}
    f_0 := \frac{1}{r^2} \frac{(\widehat{x_0-x_1}\cdot x_0)^2}{|x_1-x_0|} \qandq f_c := \frac{1}{r^2} \frac{(\widehat{x_c-x_m}\cdot x_c)^2}{|x_c-x_m|}.
\end{align*}

In order to show estimate \eqref{eq:bound_phi_ss}, the fact that $x_1$ and $x_m$ belong to $B_{r-2D}$ yields the bounds
\begin{align*}
    \min(\widehat{x_0-x_1}\cdot\widehat{x_0}, \widehat{x_c-x_m}\cdot\widehat{x_c}) \ge \frac{\sqrt{2D(2r-2D)}}{r} \ge \sqrt{\frac{2D}{r}}.
\end{align*}
Together with the obvious fact that $2+(\widehat{x_0-x_1}-\widehat{x_c-x_m})\cdot\hat\theta\ge 0$, we obtain that
\begin{align*}
    \pdrr{\varphi_2}{s} \le \frac{-r^2}{r^2-s^2} \left( \frac{1}{|x_1-x_0|}+\frac{1}{|x_c-x_m|} \right) \frac{2D}{r} \le -\frac{2D}{r^2}, 
\end{align*}
and hence the upper bound in estimate \eqref{eq:bound_phi_ss} holds. The lower bound is straightforward if we use the following inequalities $r^2-s^2\ge rD$, $|x_1-x_0|\ge D$ and $|x_m-x_c|\ge D$.

\subsection{Determinant of $H_{\varphi_2}$ at a point $(\sigma(\theta), \theta)$}
This appendix is useful in order to compute $S_2''$ at the critical points. The condition $\partial_s\varphi_2(s,\theta) = 0$ that defines $\sigma(\theta)$ can be rewritten as 
\begin{align*}
    s(2 + (\widehat{x_0-x_1}-\widehat{x_c-x_m})\cdot\hat\theta) + \sqrt{r^2-s^2} (\widehat{x_0-x_1}+\widehat{x_c-x_m})\cdot\hat\theta^\perp = 0, 
\end{align*}
which in turn allows us to obtain the relation
\begin{align}
    2 + (\widehat{x_0-x_1}-\widehat{x_c-x_m})\cdot\hat\theta = \frac{\sqrt{r^2-s^2}}{r^2}\left( 2\sqrt{r^2-s^2} - \widehat{x_0-x_1}\cdot x_0 - \widehat{x_c-x_m}\cdot x_c \right). 
    \label{eq:property}
\end{align}
Defining 
\begin{align*}
    g_0 &:= \widehat{x_0-x_1}\cdot x_0 - r^2 f_0 = \widehat{x_0-x_1}\cdot x_0 \left( 1 - \frac{\widehat{x_0-x_1}\cdot x_0}{|x_0-x_1|} \right) = - \frac{(\widehat{x_0-x_1}\cdot x_0) (\widehat{x_0-x_1}\cdot x_1)}{|x_1-x_0|}, \\
    g_c &:= \widehat{x_c-x_m}\cdot x_c - r^2 f_c = \widehat{x_c-x_m}\cdot x_c \left( 1 - \frac{\widehat{x_c-x_m}\cdot x_c}{|x_c-x_m|} \right) = -\frac{(\widehat{x_c-x_m}\cdot x_c) (\widehat{x_c-x_m}\cdot x_m)}{|x_c-x_m|},
\end{align*}
and using relation \eqref{eq:property}, we can write the second partials of $\varphi$ in the following compact way:
\begin{align*}
    \pdrr{\varphi_2}{s} = \frac{1}{r^2-s^2} \left[ g_0 + g_c - 2\sqrt{r^2-s^2} \right], \qquad \pdrr{\varphi_2}{\theta} = g_0 + g_c, \qquad \pdrt{\varphi_2}{s}{\theta} = \frac{1}{\sqrt{r^2-s^2}}( g_c-g_0 ),
\end{align*}
thus the expression of the determinant of the hessian at a point $(\sigma(\theta),\theta)$ takes the form:
\begin{align*}
    \det H_{\varphi_2} &= \frac{2}{r^2-s^2} \left( 2 g_0 g_c - \sqrt{r^2-s^2} (g_0 + g_c) \right) = \left( \frac{2g_0}{\sqrt{r^2-s^2}} - 1 \right)\left( \frac{2g_c}{\sqrt{r^2-s^2}}-1 \right)-1.
\end{align*}
This expression is zero whenever $x_1=x_m$. 

\subsection{Proof of lemma \ref{lemma:S}} \label{app:lemma}

We first claim that there is a constant such $C_3$ such that,
\begin{align} 
    |S_2'''(\theta)| \le C_3 |x_1-x_m|, \quad \theta\in\Sone.
    \label{eq:bound_Sppp}
\end{align}
Indeed, the function $(\theta, x_1, \rho, \alpha)\mapsto S_2'''(\theta, x_1, x_1+\rho\hat\alpha)$ is smooth with respect to all its arguments, and such that $\lim_{\rho\to 0} S_2''' = 0$, so we have that
\begin{align*}
    S_2'''(\theta, x_1, x_1+\rho\hat\alpha) = \rho\int_0^1 \pdr{S_2'''}{\rho}(\theta, x_1, x_1 + u\rho\hat\alpha)\ du,
\end{align*}
and since $\pdr{S_2'''}{\rho}$ is continuous on a compact set, it is uniformly bounded by some constant $C_3$, hence the estimate \eqref{eq:bound_Sppp}. 
Using a Taylor expansion of $S_2''(\theta)$ about $\theta_c \in\Theta_c$, we thus have:
\begin{align*}
    |S_2''(\theta) - S_2''(\theta_c)| \le |(\theta-\theta_c)| \int_0^1 |S_2'''(u\theta + (1-u)\theta_c)|\ du \le |\theta-\theta_c| C_3 |x_1-x_m|.
\end{align*}
Hence for $|\theta-\theta_c|\le \frac{|S_2''(\theta_c)|}{2 C_3 |x_1-x_m|}$, we have that
\begin{align*}
    |S_2''(\theta)|\ge \frac{1}{2}|S_2''(\theta_c)| = \frac{1}{2} \frac{|x_0-x_c|}{|x_1-x_0|+|x_c-x_m|} |x_1-x_m| \ge \frac{\sqrt{2rD}}{8r} |x_1-x_m|.
\end{align*}
Hence, if we define $\delta_0 = \frac{1}{2C_3} \frac{\sqrt{2rD}}{4r}\le \frac{1}{2C_3} \min_{\theta_c \in \Theta_c} \frac{|S_2''(\theta_c)|}{|x_1-x_m|}$, condition \eqref{eq:bound_Spp} is fulfilled with $C_2 := \frac{1}{8}\sqrt{\frac{2D}{r}}$.

We now focus on proving inequality \eqref{eq:bound_Sp}. Let us define the function 
\begin{align*}
    G(s,\theta) = (r^2-s^2) (\partial_s\varphi_2)^2 + (\partial_\theta \varphi_2)^2 = f_0^2 + f_c^2, 
\end{align*}
and notice that $|S_2'(\theta)| = G(\sigma(\theta), \theta)^{\frac{1}{2}}$. Let us fix $\theta$ such that $\min_{\theta_c\in\Theta_c} |\theta-\theta_c| \ge \frac{\delta_0}{2}$. By virtue of bounds \eqref{eq:bound_sigma}, there exists $\lambda\in[0,1]$ such that $\sigma(\theta) = (\lambda x_1 + (1-\lambda) x_m)\cdot\hat\theta^\perp$, so that 
\begin{align*}
    x_1-x_0 &= [(1-\lambda)(x_1-x_m)\cdot\hat\theta^\perp]\hat\theta^\perp + [\sqrt{r^2-\sigma^2} + x_1\cdot\hat\theta]\hat\theta, \\
    x_c-x_m &= [\lambda(x_1-x_m)\cdot\hat\theta^\perp]\hat\theta^\perp + [\sqrt{r^2-\sigma^2} - x_m\cdot\hat\theta]\hat\theta.
\end{align*}
Treating $f_0$ first, we split cases: if $(x_1-x_0)\cdot\hat\theta\le 0$, then necessarily $|\sigma(\theta)|\ge D$, and $\sigma(\theta)$ and $(x_1-x_0)\cdot\hat\theta^\perp$ have opposite signs. Hence,
\begin{align*}
    |f_0| = (1-\widehat{x_1-x_0}\cdot\hat\theta)\frac{|\sigma|}{r} + |\widehat{x_1-x_0}\cdot\hat\theta^\perp|\sqrt{1-\frac{\sigma^2}{r^2}} \ge \frac{D}{r}. 
\end{align*}
Now if $(x_1-x_0)\cdot\hat\theta\ge 0$, we have
\begin{align*}
    |f_0| &= \frac{1}{|x_1-x_0|} \left|(x_1-x_0-|x_1-x_0|\hat\theta)\cdot\widehat{\partial_s x_0}\right| = \frac{(1-\lambda) |(x_1-x_m)\cdot\hat\theta^\perp|}{|x_1-x_0|} \left| \sqrt{1-\frac{\sigma^2}{r^2}} - \gamma \frac{\sigma}{r} \right|,
\end{align*}
where
\begin{align*}
    \gamma &:= \frac{|x_1-x_0|-(x_1-x_0)\cdot\hat\theta}{(x_1-x_0)\cdot\hat\theta^\perp} = \frac{(x_1-x_0)\cdot\hat\theta^\perp}{|x_1-x_0|+(x_1-x_0)\cdot\hat\theta}\le 1.
\end{align*}
Hence, if $\sigma\le 0$, then 
\begin{align*}
    \sqrt{1-\frac{\sigma^2}{r^2}} - \gamma \frac{\sigma}{r} \ge \sqrt{1-\frac{\sigma^2}{r^2}} \ge \sqrt{\frac{D}{r}},
\end{align*}
and if $\sigma\ge 0$, 
\begin{align*}
    \sqrt{1-\frac{\sigma^2}{r^2}} - \gamma \frac{\sigma}{r} = \frac{1}{\sqrt{1-\frac{\sigma^2}{r^2}} + \frac{\sigma}{r}} + (1-\gamma) \frac{\sigma}{r} \ge \frac{1}{2}.
\end{align*}
Finally, we have that 
\begin{align*}
    |(x_1-x_m)\cdot\hat\theta^\perp| = |x_1-x_m||\sin(\theta_{1m}-\theta)| \ge |x_1-x_m|\sin\frac{\delta_0}{2}.
\end{align*}
In summary, denoting $C := \min \left( \frac{1}{2}, \sqrt{\frac{D}{r}} \right) \sin\frac{\delta_0}{2}$, we obtain the bound
\begin{align*}
    |f_0| \ge \min\left( \frac{D}{r}, (1-\lambda) |x_1-x_m| C \right). 
\end{align*}
Likewise, we obtain the bound 
\begin{align*}
    |f_c| \ge \min\left( \frac{D}{r}, \lambda |x_1-x_m| C \right),
\end{align*}
which implies that 
\begin{align*}
    |S_2'(\theta)| = \sqrt{f_0^2 + f_c^2} \ge \min\left( \frac{D}{r}, \frac{C}{\sqrt{2}}|x_1-x_m| \right), 
\end{align*}
where we used the inequality $\lambda^2 + (1-\lambda)^2 \ge \frac{1}{2}$ for every $\lambda\in [0,1]$. Lemma \ref{lemma:S} is proved.    

\begin{remark}
    Although we do not know the function $S_2$ explicitly, we know that we can not expect a uniform bound of the form 
    \begin{align*}
	|S_2''(\theta)| \ge C |x_1-x_m|, \quad\theta\in\Sone,
    \end{align*}
    (which in turn would give us a less singular bound of the form $|\beta^\omega|\le C |x_1-x_m|^{-\frac{1}{2}}$ by virtue of \cite[Corollary 2]{AKC}). Indeed, since $S_2'$ has two zeros and is periodic, its derivative has at least two zeros. In addition, away from the critical points, as $|x_1-x_m|\to 0$, $S_2'$ cannot decay to zero slower than $|x_1-x_m|$. This is because we can obtain an upper bound for $|S_2'|$ similar to \eqref{eq:bound_Sppp} for $|S_2'''|$. This in turn guarantees that $|S_2'|\to 0$ at least as fast as $|x_1-x_m|$ does. In that sense, lemma \ref{lemma:S} cannot be improved.
\end{remark}

\section*{Acknowledgment} This work was partially funded by the NSF under Grant DMS-0804696.


\begin{thebibliography}{99}

    \bibitem{A} {\sc V.I.~Arnol'd}, {\em Remarks on the stationary phase method and coxeter numbers}, 1973 Russ. Math. Surv. {\bf 28} 19
	
    \bibitem{AKC} {\sc G.I.~Arkhipov, A.A.~Karatsuba and V.N.Chubarikov}, {\em Trigonometric integrals}, Izv. Akad. Nauk SSSR Ser. Mat. {\bf 43}:5 (1979), 971-1003; English transl., Math. USSR-Izv. {\bf 15}:2 (1980), 211-239. 

    \bibitem{B1} {\sc G.~Bal}, {\em Inverse transport from angularly averaged measurements and time harmonic isotropic sources, in Mathematical Methods in Biomedical Imaging and Intensity-Modulated Radiation Therapy}, Ed. Y. Censor, M. Jiang, A.K. Louis, CRM Series, Scuola Normale Superiore Pisa, Italy, pp. 19-35, 2008

    \bibitem{B2} {\sc G.~Bal}, {\em Inverse Transport Theory and applications} (review paper), Inverse Problems, {\bf 25}, 055006, 2009.

    \bibitem{BJ} {\sc G.~Bal and A.~Jollivet}, {\em Time-dependent angularly averaged inverse transport}, Inverse Problems, {\bf 25}, 075010, 2009.

    \bibitem{BJLM} {\sc G.~Bal, A.~Jollivet, I.~Langmore and F.~Monard}, {\em Angular average of time-harmonic transport solutions}, Comm. Partial Differential Equations, 1532-4133, Vol. 36, Issue 6, 2011, pp. 1044-1070.	

    \bibitem{BLM} {\sc G.~Bal, I.~Langmore and F.~Monard}, {\em Inverse transport with isotropic sources and angularly averaged measurements}, Inverse Probl. Imaging, 2(1), pp. 23-42, 2008

    \bibitem{dlen6} {\sc R.~Dautray and J.-L. Lions}, {\em {Mathematical Analysis and Numerical Methods for Science and Technology. Vol.6}}, Springer-Verlag Berlin, 1993.

    \bibitem{H} {\sc L. H\" ormander}, {\em The Analysis of Linear Partial Differential Operators I: Distribution Theory and Fourier Analysis}, Springer-Verlag Berlin, 1990.

    \bibitem{mokhtar} {\sc M.~Mokhtar-Kharroubi}, {\em {Mathematical Topics in Neutron Transport  Theory}}, World Scientific, Singapore, 1997


    \bibitem{natt} {\sc F.~Natterer}, {\em The Mathematics of Computerized Tomography}, SIAM, 2001.

    \bibitem{P1} {\sc D.A.~Popov}, {\em Estimates with constants for some classes of oscillatory integrals}, Russian Math. Surveys, {\bf 52}:1, 73-145.

    \bibitem{P2} {\sc D.A.~Popov}, {\em Remarks on uniform combined estimates oscillatory integrals with simple singularities}, Izvestiya: Mathematics {\bf 72}:4, 793-816.

    \bibitem{RBH} {\sc K.~Ren, G.~Bal and A.H.~Hielscher}, {\em Frequency-domain optical tomography based on the equation of radiative transfer}, SIAM J. Sci. Comput. {\bf 28} (2006), 1463-1489.

\end{thebibliography}
\end{document}